\documentclass[10pt,a4j]{article}
\input{satoshimatome1009.STY}
\begin{document}
\title{$\scr{A}$-schemes and Zariski-Riemann spaces}
\author{Satoshi Takagi}
\date{}
\maketitle

\begin{abstract}
In this paper,
we will investigate further properties of $\scr{A}$-schemes
introduced in \cite{Takagi}.
The category of $\scr{A}$-schemes possesses
many properties of the category of coherent
schemes, and in addition,
it is co-complete and complete.
There is the universal compactification, namely,
the Zariski-Riemann space in the category
of $\scr{A}$-schemes.
We compare it with the classical Zariski-Riemann
space, and characterize the latter by a left adjoint.
\end{abstract}

\tableofcontents

\setcounter{section}{-1}
\section{Introduction}
In this paper,
we will investigate further properties of $\scr{A}$-schemes
introduced in \cite{Takagi}.
The first motivation of introducing
$\scr{A}$-schemes was to construct
a scheme-like geometrical object from various kinds of
algebraic systems, such as commutative monoids, semirings,
and etc. However,
it happens to have advantages even in the
case the algebraic system is that of rings.
The category of $\scr{A}$-schemes is 
much more flexible than that of ordinary schemes:
let us list up the properties
of $\scr{A}$-schemes, and compare with ordinary schemes:
\begin{enumerate}
\item Let $\cat{Coh.Sch}$ be the category of coherent schemes
and quasi-compact morphisms.
Then, the category $\cat{$\scr{A}$-Sch}$
of $\scr{A}$-schemes contains 
$\cat{Coh.Sch}$
as a full subcategory.
Also, $\cat{$\scr{A}$-Sch}$
is a full subcategory of the category
$\cat{LRCoh}$ of locally ringed coherent spaces
and quasi-compact morphisms (Proposition \ref{prop:asch:loc}).
\item There is a spectrum functor, and is the
left adjoint of the global section functor
$\Gamma: \cat{$\scr{A}$-Sch}^{\op} \to \cat{Rng}$
(\cite{Takagi}).
\item The inclusion functor $\cat{Coh.Sch} \to \cat{$\scr{A}$-Sch}$
preserves fiber products
(Corollary \ref{cor:preserve:fiber:prod}), and patchings via quasi-compact opens
(Proposition \ref{prop:cocomp:patch}).
\item There is a valuative criterion of separatedness
(Proposition \ref{prop:val:crit:sep})
and that of properness
(Proposition \ref{prop:val:crit:proper}).
\end{enumerate}
These imply that $\scr{A}$-schemes behave
much like ordinary schemes.
By contrast, they have more virtues than
ordinary schemes:
\begin{enumerate}
\setcounter{enumi}{4}
\item The category $\cat{$\scr{A}$-Sch}$
is small co-complete (Proposition \ref{prop:cocomp:patch})
 and small complete (Proposition \ref{prop:asch:comp}).
\textit{We don't need filteredness}.
\item There is a functorial epi-monic
decomposition of morphisms
of $\scr{A}$-schemes
(Theorem \ref{thm:PQ:decomp}). In particular,
we have the `image scheme' for each morphism.
\end{enumerate}
Therefore, we need not
distinguish pro-schemes and ind-schemes from
schemes anymore, if we work out on this category
of $\scr{A}$-schemes.
These properties give us various profits:
\begin{enumerate}
\setcounter{enumi}{6}
\item We can consider quotient $\scr{A}$-schemes
whenever there is a group action on an $\scr{A}$-scheme.
We don't need any additional condition.
\item Formal schemes can be treated 
on the same platform, as $\scr{A}$-schemes
(Example \ref{exam:formal:sch}).
\item We can think of universal `separation' of $\scr{A}$-schemes
(Proposition \ref{prop:sep:functor}).
\item We can think of universal `compactification' of $\scr{A}$-schemes,
namely, the Zariski-Riemann space
(Theorem \ref{thm:ZR:functor}).
The construction is the analog of the Stone-\v{C}ech compactification.
\end{enumerate}
The key of this extension of the category of ordinary schemes
is simple: \textit{to abandon the principal property of schemes},
namely, `locally being a spectrum of a ring'.
It is because this condition forces us only to use finite categorical operation
in the category of ordinary schemes,
and makes it very inconvenient.
In particular, we have to give up Zariski-Riemann spaces in the
category of ordinary schemes,
although it is a fairly nice locally ringed space and there are
various applications. On the other hand,
Zariski-Riemann spaces can be treated equally with
ordinary schemes, when we extend our perspective
to the category of $\scr{A}$-schemes.
Moreover, the construction of Zariski-Riemann spaces
appears to be as natural as the spectrum of a ring.

Let us describe the contents of this paper.
In \S 1, we will prove some properties of
coherent spaces, which we will need later.
In particular, a quasi-compact morphism
of coherent spaces is epic if and only if it is surjective,
and its image is closed if and only if it is
specialization closed.

In \S 2, we will discuss the properties
of the category of $\scr{A}$-schemes,
namely we will prove the co-completeness
and completeness.
The key lemma is the functorical decomposition
of morphisms. This gives us the upper bound
of the cardinality of the set of morphisms
with fixed targets, and hence enables
us to construct limits.
This is the analog of the construction
of co-limits in the category of algebras of various kinds.

In \S 3, we define separated and proper morphisms
of $\scr{A}$-schemes, and give valuative criteria
for separatedness and properness.
Unlike ordinary schemes, we don't have a canonical
morphism $\Spec \scr{O}_{X,x} \to X$
in the category of $\scr{A}$-schemes,
where $x$ is a point of an $\scr{A}$-scheme $X$.
Hence, we had to modify the testing morphisms.
The valuative criteria in the category of ordinary schemes
check the right lifting properties of the commutative square
\[
\xymatrix{
\Spec K \ar[r] \ar[d] & X \ar[d] \\
\Spec R \ar[r] & S
}
\]
where $K$ is an arbitrary field and $R$ is its valuation field.
The left vertical arrow is the `testing morphism'.
In the category of $\scr{A}$-schemes,
we must replace the testing morphisms to
formulate the valuative criteria:
namely, just take the set $\{\xi,\eta\}$ of the generic point and
the closed point of $\Spec R$ with a natural
induced $\scr{A}$-scheme structure.
Once we have the valuative criteria,
we can construct the universal separation
and the universal compactification;
the latter is treated in \S 4.
We note here, that we will not include
`of finite type' condition in the definition of proper
morphisms, for we want to take limits.
We emphasize the fact
that separated morphisms and universally closed morphisms are
closed under taking infinite limits,
while morphisms of finite types are not.

In \S 4, we construct the Zariski-Riemann space
as the universal compactification using the
adjoint functor theorem.
Later, we will also consider `classical Zariski-Riemann
space' on irreducible, reduced $\scr{A}$-schemes,
as it happens to be more easier to analyze than
the previous Zariski-Riemann space.
This construction is the analog of the conventional
one, but with different flavor: \textit{we tried
not to use valuation rings when defining it}.
This reveals the naturalness of the concept
of valuation rings--- that it is as natural as the
concept of `local rings' of spectra of rings.
Also, note that what localization is to the spectrum
is what separated, of finite type morphism is
to the Zariski-Riemann space.
These also imply that the way of 
constructing a 
topological object from rings is not at all
unique---we can consider another
`algebraic geometry'.

Here, we also compare our previous
Zariski-Riemann space with the classical one.
Actually, the topology of the classical Zariski-Riemann
space happens to be coarser:
it is, in a sense, the coarsest possible topology.
This property, which we will call `of profinite type',
characterizes the classical Zariski-Riemann space.
Though the classical Zariski-Riemann space
has a weaker universality, it is valuable
since it gives us concrete descriptions of its
structure and morphisms.
Any separated dominant morphism 
of ordinary integral schemes is of profinite type,
so that they can be embedded into
a universal proper, of profinite type $\scr{A}$-scheme.

In this paper,
we only constructed Zariski-Riemann space
for irreducible, reduced $\scr{A}$-schemes,
since this assumption makes the argument much simpler,
and it will be sufficient for most of the applications.
We believe that it is possible to extend it
to arbitrary $\scr{A}$-schemes
with a little more effort.
We have proved a variant of the Nagata embedding 
(Corollary \ref{cor:vari:nagata}).
The original version of the Nagata embedding
can also be proven, and will be shown in the forthcoming paper.
We decided not to prove it here,
since there are various proofs published already
(\cite{Nagata},\cite{Con},\cite{Tem}),
and it takes a little more detailed work
which will make this paper more longer if we include it.
However, the proof is rather natural and intuitive
than the former ones.

We summarized the definition of $\scr{A}$-schemes
at the end of this paper, as an appendix.
This will be sufficient for the reader
to go through this paper, though
he hasn't looked over \cite{Takagi}.

\subsection{Notation and conventions}
The reader is assumed to have
standard knowledge of categorical theories;
see for example, \cite{CWM}, \cite{KS}.
We fix a universe, and
all sets are assumed to be small.
The category of small sets (resp. sober spaces
and continuous maps)
is denoted by $\cat{Set}$ (resp. $\cat{Sob}$).

When we talk of an algebraic system,
all the operators are finitary,
and all the axioms are identities.
Any ring and any monoid is commutative,
and unital.
For a ring (or, other algebras with a structure
of a multiplicative monoid) $R$,
we denote by $R_{S}$ the localization
of $R$ along the multiplicative system $S$ of $R$.

For any set $\mathcal{S}$,
$\scr{P}(\mathcal{S})$ is the power set of $\mathcal{S}$,
and $\scr{P}^{f}(\mathcal{S})$ is the
set of finite subsets of $\mathcal{S}$.

When given an $\scr{A}$-scheme
(or, any topological space with its structure sheaf) $X$,
we denote the underlying space by $|X|$.

We frequently denote finite summations
by $\sum^{<\infty}$, when the range
of the index is not crucial.
The same thing can be said for
the notation $\cup^{<\infty}$ for finite
unions.

\section{Properties of coherent spaces}
In this section, we will
investigate some properties
of sober and coherent spaces.
Recall that a topological space $X$ is \textit{sober},
if every irreducible closed subset of $X$
has a unique generic point.
For a sober space $X$, $C(X)$
is the set of closed subsets of $X$.
This becomes a complete II-ring;
see appendix \S 5 for further details.
A topological space $X$ is \textit{coherent},
if it is sober, quasi-compact, quasi-separated,
and has a quasi-compact open basis.

\subsection{Monic and epic maps}
\begin{Lem}
\label{lem:eq:inj:monic}
Let $\sigma$ be any algebraic system,
and $f:A \to  B$ be a homomorphism
of $\sigma$-algebras.
Then, $f$ is monic if and only if $f$ is injective.
\end{Lem}
\begin{proof}
The `if' part is clear.

Suppose $f$ is not injective.
Then there are two distinct elements $a_{1},a_{2} \in A$
such that $f(a_{1})=f(a_{2})$.
Let $R_{0}$ be the initial object in $\cat{$\sigma $-alg}$.
Define two homomorphisms $g_{i}:R_{0}[X] \to A$
by $X \mapsto a_{i}$ for $i=1,2$.
Then, $fg_{1}=fg_{2}$ but $g_{1} \neq g_{2}$,
a contradiction.
\end{proof}

In the sequel, let $f:X \to  Y$ be a morphism of sober spaces,
and $f^{\#}:C(Y) \to C(X)$ the corresponding 
homomorphism of complete idealic rings
(for complete idealic rings, see Appendix).

\begin{Prop}
The followings are equivalent:
\begin{enumerate}[(i)]
\item $f$ is injective.
\item $f$ is monic.
\item Any prime element of $C(X)$
is in the image of $f^{\#}$.
\end{enumerate}
\end{Prop}
\begin{proof}
(i)$\Leftrightarrow$(ii) follows from Lemma \ref{lem:eq:inj:monic}.

(i)$\Rightarrow$(iii):
Let $x$ be a point of $X$.
It suffices to show that $f^{-1}(\overline{f(x)})=\overline{\{x\}}$.
Let $w \in f^{-1}(\overline{f(x)})$ be the generic point of
an irreducible component of $f^{-1}(\overline{f(x)})$.
Then we have $f(w)=f(x)$.
Since $f$ is injective, $w=x$.
This shows that $f^{-1}(\overline{f(x)})=x$.

(iii)$\Rightarrow$(i):
Note that
$\Spec C(X) \subset \Imag f^{\#}$
shows that $\overline{\{x\}}=f^{-1}(\overline{f(x)})$
for any point $x$ of $X$.
If $f(x)=f(x^{\prime})$
for two points of $X$, then
\[
\overline{\{x\}}=f^{-1}(\overline{f(x)})=f^{-1}(\overline{f(x^{\prime})})
=\overline{\{x^{\prime}\}}.
\]
Since $X$ is sober, $x$ coincides $x^{\prime}$.
\end{proof}

\begin{Prop}
\label{prop:coh:surj:homeo}
The followings are equivalent:
\begin{enumerate}[(i)]
\item $f^{\#}$ is surjective.
\item $\Imag f$ is homeomorphic to $X$.
\end{enumerate}
\end{Prop}
\begin{proof}
(i)$\Rightarrow$(ii):
Since $f^{\#}$ is surjective,
it is epic, hence $f$ is monic.
This shows that $f$ is injective.
The surjectivity of $f^{\#}$ shows that
$f^{-1}(\overline{f(z)}\cap \Imag f)=z$
for any closed subset $z$ of $X$.
Hence $\overline{f(z)}\cap \Imag f=f(z)$,
which implies $f(z)$ is closed in $\Imag f$.

(ii)$\Rightarrow$(i):
Since $f$ is injective,
$f^{-1}(f(z))=z$ for any closed subset $z$ of $X$.
On the other hand, $f(z)$ is closed in $\Imag f$
since $f$ is homeomorphic onto the image.
Therefore, $f^{-1}(\overline{f(z)})=z$,
which implies $f^{\#}$ is surjective.
\end{proof}

Next, we will investigate the
condition when a morphism $f:X \to Y$
of sober spaces becomes epic.
Let $\cat{IIRng$^{\dagger}$}$
be the catogory of complete idealic semirings
(see Appendix).
First, note that the functor
$C:\cat{Sob}^{\op} \to \cat{IIRng$^{\dagger}$}$
is fully faithful,
since it is the right adjoint and left inverse of $\Spec$
(\cite{Takagi}).
An object $R$ of $\cat{IIRng$^{\dagger}$}$
is \textit{spatial}, if $R$ is isomorphic to $C(X)$
for some sober space $X$.

\begin{Rmk}
There exists some non-spatial complete II-rings.
Let $C(\RR)$ be the complete II-rings
of closed subsets of the real line with the standard topology.
Let $R=C(\RR)/\equiv$ be the quotient complete II-ring,
where $\equiv$ is the congruence generated by
$Z=\overline{Z^{o}}$ for any $Z$,
where $Z^{o}$ is the open kernel of $Z$.
Then, $R$ is a non-trivial II-ring,
but has no points; see \cite{Steven}.
\end{Rmk}

\begin{Lem}
\label{lem:equiv:epi:sob:inj}
Let $R$ be an object of
$\cat{IIRng$^{\dagger}$}$,
and $R[t]$ be a polynomial complete idealic semiring
with idempotent multiplication.
\begin{enumerate}
\item An element of $R[t]$
can be expressed by $a+bt$,
where $a,b$ are elements of $R$
and $a \leq b$.
\item A prime element $p$ of $R[t]$
is either of the following:
\begin{enumerate}[(a)]
\item $p=a+bt$, where $a,b$ are prime elements of $R$
and $a \leq b$.
\item $p=a+t$, where $a$ is a prime element.
\end{enumerate}
\item If an object $R$ of $\cat{IIRng$^{\dagger}$}$
is spatial, then so is $R[t]$.
\item In particular,
a morphism $f:X \to Y$ of sober spaces
is epic if and only if $f^{\#}:C(Y) \to C(X)$
is monic.
\end{enumerate}
\end{Lem}
\begin{proof}
\begin{enumerate}
\item Easy.
\item Let $p=a+bt$ be a prime element.
It is easy to see that $a \in R$ must be prime.
Also, $b$ must be prime or $1$,
since $xt \cdot yt=xyt$.
\item Easy.
\item The `if' part is obvious,
since $\cat{Sob}^{\op}$ is can be regarded
as a full subcategory of $\cat{IIRng$^{\dagger}$}$
via $C$.
If $f^{\#}$ is not monic,
we have two distinct morphisms $g,h:\FF_{1}[t] \to C(Y)$
such that $f^{\#}g=f^{\#}h$,
where $\FF_{1}$ is the initial object of $\cat{IIRng$^{\dagger}$}$.
But since $\FF_{1}[t]$
is spatial, we conclude that $f$ is not epic.
\end{enumerate}
\end{proof}

\begin{Prop}
\label{prop:epic:top}
The followings are equivalent:
\begin{enumerate}
\item $f$ is epic.
\item $f^{\#}$ is injective.
\item $f^{\#}$ is monic (in $\cat{IIRng$^{\dagger}$}$).
\item $\Imag f \cap z$ is dense in $z$,
for any closed subset $z$ of $Y$.
\end{enumerate}
\end{Prop}
\begin{proof}
(i)$\Leftrightarrow$(ii)$\Leftrightarrow$(iii) is a consequence of 
Lemma \ref{lem:eq:inj:monic} and \ref{lem:equiv:epi:sob:inj}.

(ii)$\Rightarrow$(iv):
For any closed subset $z$ of $X$,
let $z^{\prime}$ be the closure
of $\Imag f \cap z$.
Then we have $f^{\#}(z^{\prime})=f^{\#}(z)$.
Since $f^{\#}$ is injective, we have
$z^{\prime}=z$, hence the result follows.

(iv)$\Rightarrow$(ii):
Suppose $f^{\#}(z)=f^{\#}(z^{\prime})$
for some closed subsets $z,z^{\prime}$ of $Y$.
The equation $ff^{-1}(z)=\Imag f \cap z$
induces
\[
z=\overline{\Imag f \cap z}=\overline{\Imag f \cap z^{\prime}}
=z^{\prime}.
\]
\end{proof}

\subsection{Coherent spaces}
\begin{Prop}
Let $f:X \to Y$ be an epimorphism of sober spaces,
and $X$ be noetherian.
Then $f$ is surjective.
\end{Prop}
\begin{proof}
Let $y$ be any point of $Y$, and
$Z=\overline{\{y\}}$ be the irreducible subset
corresponding to $y$.
Since $X$ is noetherian,
$f^{-1}(Z)$ can be covered by a finite number
of irreducible closed subsets:
$f^{-1}(Z)=\cup_{i}^{<\infty}W_{i}$.
Let $\xi_{i}$ be the generic point of $W_{i}$ for each $i$.
Since the image of $f$ is dense in $Z$
and $Z$ is irreducible, at least one
of the $\xi_{i}$'s must be mapped to $y$.
\end{proof}

The proof of the next theorem
requires some preliminaries
on ultrafilters (see \cite{CN}, for example).
The reader who knows well may skip
and go on to the next theorem.
\begin{Def}
Let $\mathcal{S}$ be a non-empty set.
\begin{enumerate}
\item A \textit{filter} $\scr{F}$ on $\mathcal{S}$
is a non-empty subset of $\scr{P}(\mathcal{S})$ satisfying:
\begin{enumerate}[(i)]
\item $\emptyset \notin \scr{F}$. 
\item If $A \in \scr{F}$ and $A \subset B$, then $B \in \scr{F}$.
\item If $A$ and $B$ are elements of $\scr{F}$,
then $A \cap B \in \scr{F}$.
\end{enumerate}
The set of filters becomes a poset by inclusions.
\item A maximal filter with respect to inclusions is called an \textit{ultrafilter}.
\end{enumerate}
\end{Def}
A filter $\scr{U}$ is an ultrafilter if and only if
$a \in \scr{U}$ or $a^{c} \in \scr{U}$ for any subset $a$ of $\mathcal{S}$.
Also, note that exactly one of $a$ or $a^{c}$ is in $\scr{U}$.

For any $s \in \mathcal{S}$,
\[
\scr{U}_{s}=\{ a \subset \mathcal{S} \mid s \in a\}
\]
becomes an ultrafilter, which is called \textit{principal}.
An ultrafilter is principal, if and only if it contains a finite subset of $\mathcal{S}$.

Let $\scr{S}$ be a subset of $\scr{P}(\mathcal{S})$
satisfying
\begin{enumerate}[(i)]
\item $\emptyset \notin \scr{S}$.
\item If $A$  and $B$ are elements of $\scr{S}$,
then $A \cap B \in \scr{S}$.
\end{enumerate}
Then there exists a ultrafilter containing $\scr{S}$,
using the axiom of choice.

A filter $\scr{F}$ on a non-empty set is \textit{prime},
if $a \cup b \in \scr{F}$ implies either $a \in \scr{F}$
or $b \in \scr{F}$. One can easily prove
that the notion of prime filters is equivalent to 
that of ultrafilters.

\begin{Lem}
\label{lem:coprod:ultrafilter}
Let $X=\{x_{\lambda}\}_{\lambda}$ be a set, and
$X_{\lambda}=\{x_{\lambda}\}$
be one-pointed spaces, regarded as coherent spaces.
Let $X_{\infty}=\amalg_{\lambda} X_{\lambda}$
be the coproduct of $X_{\lambda}$'s in the category
of coherent spaces.
Then, any point of $X_{\infty}$
corresponds to a ultrafilter on $X$.
\end{Lem}
\begin{proof}
First, note that $X$ is isomorphic to $\Spec (\prod_{\lambda}\FF_{1})$,
where $\FF_{1}$ is the initial object in the category of $\cat{IIRng}$.
Hence, a closed subset of $X_{\infty}$ corresponds to a filter on $X$,
and any point of $X_{\infty}$ corresponds to
a prime filter, in other words, a ultrafilter on $X$.
\end{proof}

\begin{Thm}
\label{thm:alg:epi:surj}
Let $f:X \to Y$ be an epimorphism of coherent spaces.
Then $f$ is surjective.
\end{Thm}
\begin{proof}
Let $y_{0} \in Y$ be any point of $Y$,
and $Y_{0}$ be the closure of $\{y_{0}\}$ in $Y$.
Since $f$ is epic, 
$\{y_{\lambda}\}_{\lambda}=\Imag f \cap Y_{0}$ is dense in $Y_{0}$.
Assume that $y_{0} \notin \Imag f \cap Y_{0}$.
Then, $\Imag f \cap Y_{0}$ must be an infinite set,
since if it is finite, then its closure is equal to
$\cup_{\lambda}^{<\infty}\overline{\{y_{\lambda}\}}$,
which is a proper closed subset of $Y_{0}$.
Choose $x_{\lambda} \in X$ such that $f(x_{\lambda})=y_{\lambda}$
for each $\lambda$, and set $S=\{x_{\lambda}\}_{\lambda}$.
Also, let $\tilde{X}=\amalg_{\lambda}\{x_{\lambda}\}$
be the coproduct of $\{x_{\lambda}\}$'s in the category
of coherent spaces. 
By Lemma \ref{lem:coprod:ultrafilter},
the points of $\tilde{X}$
correspond to the ultrafilters on $S$.
We have the natural commutative diagram
\[
\xymatrix{
\tilde{X} \ar[d]_{\iota} \ar[dr]^{\tilde{f}} & \\
X \ar[r]_{f} & Y
}
\]
Next, we define a map
$\varphi:C(Y_{0})\setminus \{Y_{0}\} \to \scr{P}S \setminus \{\emptyset\}$
by sending $Z$ to $\{x \in S \mid f(x) \notin Z\}$.
This is well defined, since $f(S)$ is dense in $Y_{0}$.
Also, note that $\Imag \varphi$ is stable under taking finite intersections:
$\varphi(Z_{1}) \cap \varphi(Z_{2})=\varphi(Z_{1} \cup Z_{2})$.
It follows that there is an ultrafilter $\scr{U}$ on $S$, containing
the image of $\varphi$.
Let $x_{0}$ be the point of $\tilde{X}$, corresponding to $\scr{U}$.
We claim that $\tilde{f}(x_{0})=y_{0}$,
from which $f(\iota(x_{0}))=y_{0}$ follows.
Indeed, the image of $x_{0}$ is the generic point
of the intersections of  $Z \in C(X)_{\cpt}$ 
such that $x_{0} \in \tilde{f}^{-1}(Z)$.
Hence, it suffices to show that $Y_{0}$ is the only closed
subset satisfying the condition. If $Z \neq Y_{0}$,
then $Z$ is a proper closed subset of $Y_{0}$,
and $x_{0} \in \tilde{f}^{-1}(Z)$ implies $\varphi(Z)^{c} \in \scr{U}$.
But on the other hand, $\varphi(Z) \in \scr{U}$,
which is a contradiction to $\scr{U}$ being a ultrafilter.
\end{proof}
\begin{Cor}
Let $f:X \to Y$ be a morphism of coherent spaces.
Then, the image of $f$ with its induced topology
is also coherent.
\end{Cor}
\begin{proof}
The morphism $f$ corresponds to the homomorphism
$f^{\#}:C(Y)_{\cpt} \to C(X)_{\cpt}$ of II-rings.
Let $R$ be the image of $f^{\#}$.
Then, $C(Y)_{\cpt} \to R$ is surjective, and
$R \to C(X)_{\cpt}$ is injective. Set $W=\Spec R$.
Then, $f$ factors through $W$, with $X \to W$ epic
(hence surjective by Theorem \ref{thm:alg:epi:surj})
and $W \subset Y$ an immersion by Proposition \ref{prop:coh:surj:homeo}.
This tells that $W$ coincides with the image of $f$.
\end{proof}

\begin{Exam}
\begin{enumerate}
\item
Let $Y=|\Aff^{1}_{\CC}|$
be the underlying space of the affine line over $\CC$,
and $X=Y(\CC)$ be the set of closed points
of $Y$, endowed with a discrete topology.
Then, the natural map $X \to Y$ is
a morphism of sober spaces.
This is an epimorphism
by Proposition \ref{prop:epic:top},
but not surjective, since
the image does not contain the generic point of $Y$.

On the other hand, we have a morphism
$\alg(X) \to Y$ of coherent spaces:
here, $\alg$ is the left adjoint of 
the underlying functor $\cat{Coh} \to \cat{Sob}$
(\cite{Takagi}).
In this case, $\alg(X)$ is just the coproduct
$\amalg_{x \in X} \{x\}$
in the category of coherent spaces.
This is also epic, hence surjective
by \ref{thm:alg:epi:surj}.
The non-principal points of $\alg(X)$
maps onto the generic point of $Y$.
\item
Let $X$ be as above,
and set $V=|\Aff^{2}_{\CC}|$,
the affine plane over $\CC$.
Since there is a non-canonical bijection
$\varphi:\CC \to \CC^{2}$,
there exists a map $\varphi:X \to V$,
the image of which is the set
of closed points of $V$.
When we algebraize $X$,
we again obtain a surjective
map $\alg(X) \to V$.
Some of the non-principal points
of $X$ map to a generic point
of a curve on $V$,
others map to the generic point of $V$.
These two examples tell us that
the non-principal points of $\alg(X)$
behave like `universal generic points'
of $X$, although the Krull dimension
of $\alg(X)$ is zero.
\end{enumerate}
\end{Exam}

The next theorem is important when
we consider valuative criteria.
\begin{Thm}
\label{thm:dominant:image:minimal}
Let $f:X \to Y$ be a dominant morphism of coherent spaces.
Then, any minimal point of $Y$ (that is, the generic point of 
an irreducible component of $Y$) is contained in the
image of $f$.
\end{Thm}
\begin{proof}
Let $y_{0}$ be any minimal point of $Y$,
and $\{U^{\lambda}\}_{\lambda}$ be the filtered
system of quasi-compact open neighborhood of $y_{0}$.
For each $\lambda$, $f^{-1}(U^{\lambda}) \to U^{\lambda}$ is dominant
since $f$ is so.
Set $U^{\infty}=\underleftarrow{\lim}_{\lambda}U^{\lambda}$.
This is a pointed space $\{y_{0}\}$,
since the underlying functor $\cat{Coh} \to \cat{Set}$
preserves limits (\cite{Takagi}).
We claim that $f^{-1}(U^{\infty}) \to U^{\infty}$
is dominant (in particular, $f^{-1}(U^{\infty}) \neq \emptyset$).
Since $C(U^{\infty})_{\cpt}$ and $C(f^{-1}(U^{\infty}))_{\cpt}$
are naturally isomorphic to $\underrightarrow{\lim} C(U^{\lambda})_{\cpt}$
and $\underrightarrow{\lim} C(f^{-1}(U^{\lambda}))_{\cpt}$
respectively, it suffices to show that the homomorphism
\[
f^{\#}:\underrightarrow{\lim} C(U^{\lambda})_{\cpt} \to 
\underrightarrow{\lim} C(f^{-1}(U^{\lambda}))_{\cpt}
\]
satisfies $f^{\#}(Z)=0 \Rightarrow Z=0$.
Let $Z$ be an element satisfying $f^{\#}(Z)=0$.
Since $\{C(U^{\lambda})_{\cpt}\}_{\lambda}$ is a filtered
inductive system, $Z$ and $f^{\#}(Z)$ can be represented by an element
of $C(U^{\lambda})_{\cpt}$ and $C(f^{-1}(U^{\lambda}))_{\cpt}$
for some $\lambda$, respectively.
Since $C(U^{\lambda})_{\cpt} \to C(f^{-1}(U^{\lambda}))_{\cpt}$
is dominant, $Z$ must be zero.
\end{proof}

\begin{Cor}
\label{cor:image:close:spe}
Let $X$ be a coherent subspace
of a coherent space $Y$. Then,
the closure of $X$ consists of all points
which are specializations of points on $X$.
\end{Cor}
\begin{proof}
Let $Z$ be the set of points which
are specializations of points on $X$.
It is clear that $Z \subset \overline{X}$,
so we will show the converse.
Let $y$ be any point in the closure of $X$ in $Y$.
Since $\overline{X}$ is also a coherent subspace,
there exists a minimal point $y_{0}$ of $\overline{X}$
which is a generalization of $y$.
Since $X \to \overline{X}$ is dominant,
$y_{0}$ is contained in $X$ by Theorem
\ref{thm:dominant:image:minimal}.
Since $Z$ is stable under specializations, $y$
must be in $Z$.
\end{proof}

\section{The category of $\scr{A}$-Schemes}

In \cite{Takagi}, we introduced
the definition of $\scr{A}$-schemes.
The advantage of the notion of $\scr{A}$-schemes
is not only generalizing the concept of schemes to other
algebraic systems, 
but also giving the way to infinite categorical
operations:
in fact, the category of $\scr{A}$-schemes
is small complete and co-complete.
A finite patching over quasi-compact open
sets, and fiber products commute
with those of $\scr{Q}$-schemes, namely,
ordinary schemes.

\textit{Notations:} from now on,
the homomorphism $C(Y)_{\cpt} \to C(X)_{\cpt}$
associated to a morphism $f:X \to Y$ of
$\scr{A}$-schemes is denoted by $f^{-1}$, or $|f|^{-1}$.
We use the notation $f^{\#}$ for
the morphism of structure sheaves
$\scr{O}_{Y} \to f_{*}\scr{O}_{X}$.

\subsection{Co-completeness}
First, we begin with describing
what $\scr{A}$-schemes is like
when the algebraic system is that of rings.
\begin{Prop}
\label{prop:asch:loc}
Let $\sigma$ be an algebraic system of rings,
and $\cat{LRCoh}$ be the category
of locally ringed coherent spaces and quasi-compact morphisms.
Then, there is a natural
underlying functor
$\cat{$\scr{A}$-Sch} \to \cat{LRCoh}$
defined by $(X,\scr{O}_{X},\beta_{X}) \mapsto (X,\scr{O}_{X})$.
Further, this functor is fully faithful.
\end{Prop}
\begin{proof}
Let $X=(X,\scr{O}_{X},\beta_{X})$
be an $\scr{A}$-scheme.
What we have to show first is that
$X$ is a locally ringed space,
i.e. $\scr{O}_{X,x}$ is a local ring for any $x \in X$.
Let $\mathfrak{M}_{x}$ be a subset
of $\scr{O}_{X,x}$, consisting of germs
$a$ such that $\beta_{X}\alpha_{2}(a) \ni x$.
We will show that this is the unique maximal ideal
of $\scr{O}_{X,x}$.

Let $a,b$ be two elements of $\mathfrak{M}_{x}$.
We have $\alpha_{2}(a+b) \leq \alpha_{2}(a)+\alpha_{2}(b)$;
recall that $\alpha_{2}(a)$ is the principal ideal
generated by $a$.
Hence, 
\[
\beta_{X}\alpha_{2}(a+b)
\leq \beta_{X}\alpha_{2}(a)+\beta_{X}\alpha_{2}(b)
\leq x+x=x,
\]
which shows that $a+b \in \mathfrak{M}_{X}$.
It is easier to show that
$ca \in \mathfrak{M}_{x}$ for any $c \in \scr{O}_{X,x}$
and $a \in \mathfrak{M}_{x}$.
Hence, $\mathfrak{M}_{x}$ is an ideal of $\scr{O}_{X,x}$.

Suppose $\langle U,a \rangle \in \scr{O}_{X,x}$ is not contained
in $\mathfrak{M}_{x}$.
Set $Z=\beta_{X}\alpha_{2}(a)$.
This does not contain $x$.
Set $V=U\setminus Z$.
Since restriction morphisms
reflect localizations (see appendix \S 5 for the terminology), $a|_{V}$ is invertible,
hence $a$ is invertible in $\scr{O}_{X,x}$.
This shows that $\scr{O}_{X,x}$ is local,
the maximal ideal of which is $\mathfrak{M}_{x}$.

Let $f:X \to  Y$ be a morphism
of $\scr{A}$-schemes.
It suffices to show that 
$f^{\#}:\scr{O}_{Y, f(x)} \to \scr{O}_{X, x}$
is a local homomorphism for any $x \in X$.
Let $a$ be an element of $\mathfrak{M}_{y}$,
where $y=f(x)$. Then,
\[
\beta_{X}\alpha_{2}(f^{\#}(a))
=\beta_{X}(\alpha_{1}f)(\alpha_{2}(a))
=|f|^{-1}\beta_{Y}\alpha_{2}(a).
\]
Since $\beta_{Y}\alpha_{2}(a) \leq y$,
we have $|f|^{-1}\beta_{X}\alpha_{2}(a)\leq x$,
which shows that
$\mathfrak{M}_{y} \subset (f^{\#})^{-1}\mathfrak{M}_{x}$.

Hence, we have a functor
$U:\cat{$\scr{A}$-Sch} \to \cat{LRCoh}$.
It remains to show that this functor is fully faithful.
Let $X, Y$ be two $\scr{A}$-schemes,
and $f=(|f|,f^{\#}):UX \to UY$ be a morphism
of locally ringed spaces.
It suffices to show that the following diagram
\[
\xymatrix{
\alpha_{1}\scr{O}_{Y}  \ar[d]_{\beta_{Y}} \ar[r]^{\alpha_{1}f^{\#}} 
& |f|_{*}\alpha_{1}\scr{O}_{X} \ar[d]^{|f|_{*}\beta_{X}} \\
\tau_{Y} \ar[r]_{|f|^{-1}} & |f|_{*}\tau_{X}
}
\]
is commutative.
Let $\mathfrak{a}$ be a section of $\alpha_{1}\scr{O}_{Y}$.
Then
\[
|f|^{-1}\circ \beta_{Y}(\mathfrak{a})
=\{ x \mid \mathfrak{M}_{Y,f(x)} \supset \mathfrak{a}\}, \quad
f_{*}\beta_{X}\circ \alpha_{1}f^{\#}(\mathfrak{a})
=\{ x \mid \mathfrak{M}_{X,x} \supset f^{\#}\mathfrak{a}\},
\]
where $\mathfrak{M}_{Y,f(x)}$ and $\mathfrak{M}_{X,x}$
are the maximal ideals of $\scr{O}_{Y,f(x)}$,
$\scr{O}_{X,x}$, respectively.
But the right-hand sides of the both equations coincide,
since $f^{\#}:\scr{O}_{Y,f(x)} \to \scr{O}_{X,x}$
is a local homomorphism for any $x$.
\end{proof}

In the sequel, we fix
a schematizable algebraic type
$\scr{A}=(\sigma,\alpha_{1},\alpha_{2},\gamma)$
(\cite{Takagi}).
\begin{Def}
\begin{enumerate}
\item
An $\scr{A}$-scheme $X$ is a \textit{$\scr{Q}$-scheme},
if it is locally isomorphic to $\Spec^{\scr{A}} R$,
for some $\sigma$-algebra $R$.
\item
Let $\cat{$\scr{Q}$-Sch}$ be the full subcategory
of $\cat{$\scr{A}$-Sch}$, consisting
of $\scr{Q}$-schemes.
\end{enumerate}
\end{Def}
Proposition \ref{prop:asch:loc} tells that,
if $\sigma$ is the algebraic system of rings,
then $\cat{$\scr{Q}$-Sch}$ is the category
of coherent schemes and quasi-compact morphisms,
since a morphism of schemes is a morphism
of locally ringed spaces.
\begin{Prop}
The category $\cat{$\scr{Q}$-Sch}$
admits finite patching via quasi-compact opens,
namely:
let $X_{1},\cdots ,X_{n}$ be $\scr{Q}$-schemes,
$U_{ij} \subset X_{i}$ be quasi-compact open subsets,
and $\varphi_{ji}:U_{ij} \to U_{ji}$
be isomorphisms satisfying
$\varphi_{kj}\circ \varphi_{ji}=\varphi_{ki}$
on $U_{ij} \cap U_{jk}$.
Then, there exists a co-equalizer $X$
of  $\amalg_{ij} U_{ij} \rightrightarrows \amalg_{i} X_{i}$,
such that $\amalg_{i} X_{i} \to X$ is a quasi-compact
open covering.
\end{Prop}
Note that, when we speak of a covering,
it must be always surjective.
\begin{proof}
By induction on $n$,
it suffices to prove for $n=2$:
let $X_{1}$, $X_{2}$ be two $\scr{Q}$-schemes,
and $X_{1} \hookleftarrow U \hookrightarrow X_{2}$
be the intersection quasi-compact open subscheme of $X_{1}$ and $X_{2}$.
Let $X$ be the amalgamation $X_{1} \amalg_{U} X_{2}$
of $X_{1}$ and $X_{2}$ along $U$.
This is well defined, and coincides
with the usual topology, since $X$
is coherent, thanks to $U$ being quasi-compact.
This also shows that $\{X_{i} \to X\}_{i=1,2}$
is indeed a covering.
\end{proof}

The next lemma is peculiar to
II-rings.
\begin{Lem}
Let $C^{\bullet}=\{C^{\lambda}\}$ be a projective system of
II-rings. Let $C$ be the limit of $C^{\bullet}$,
and $\pi_{\lambda}:C \to C^{\lambda}$ be the natural morphisms.
Then, for any $Z \in C$, the localization $C_{Z}$ along $Z$
is naturally isomorphic to 
$\underleftarrow{\lim}_{\lambda}(C^{\lambda})_{\pi_{\lambda}Z}$.
\end{Lem}
\begin{proof}
Since $Z$ maps to $1$ by the morphism 
$C \to (C_{\lambda})_{\pi_{\lambda}Z}$,
We have a natural morphism $C_{Z} \to (C^{\lambda})_{\pi_{\lambda}Z}$.
This in turn gives a natural morphism
$C_{Z} \to \underleftarrow{\lim}_{\lambda}(C^{\lambda})_{\pi_{\lambda}Z}$.

Next, we see that if $\varphi:C_{Z} \to C_{Z}$
is a endomorphism such that $\pi_{\lambda}\varphi=\pi_{\lambda}$
for any $\lambda$, then $\varphi$ is the identity.
$\pi_{\lambda}\varphi(x)=\pi_{\lambda}x$ in 
$(C^{\lambda})_{\pi_{\lambda}Z}$ is equivalent to
$\pi_{\lambda}(Z\varphi(x))=\pi_{\lambda}(Zx)$ in $C^{\lambda}$.
Since this holds for any $\lambda$,
we have $Z\varphi(x)=Zx$, which is equivalent to $\varphi(x)=x$
in $C_{Z}$.

Finally, we construct 
$\underleftarrow{\lim}_{\lambda}(C^{\lambda})_{\pi_{\lambda}Z}
\to C_{Z}$. The map $f:\underleftarrow{\lim}_{\lambda}C^{\lambda}
\to C_{Z}$ is already defined, hence
we only need to verify that $f(x)=f(y)$ implies $x=y$ in
$\underleftarrow{\lim}_{\lambda}(C^{\lambda})_{\pi_{\lambda}Z}$,
but this is obvious.

Combining all the arguments,
we see that $C_{Z}$ coincides with
$\underleftarrow{\lim}_{\lambda}(C^{\lambda})_{\pi_{\lambda}Z}$.
\end{proof}

\begin{Prop}
\begin{enumerate}
\label{prop:limit:tau}
\item Let $\{X_{\lambda}\}$ be an inductive system
of coherent spaces, and $X=\underrightarrow{\lim}_{\lambda}X_{\lambda}$.
Then, there is a natural isomorphism
$\tau_{X} \simeq \underleftarrow{\lim}_{\lambda}
\iota_{\lambda *} \tau_{X_{\lambda}}$,
where $\iota_{\lambda}:X_{\lambda} \to X$ are the
induced morphisms.
\item Let $\{X^{\lambda}\}$ be a filtered projective system
of coherent spaces, and $X=\underleftarrow{\lim}_{\lambda}X^{\lambda}$.
Then, there is a natural isomorphism
$\tau_{X} \simeq \underrightarrow{\lim}_{\lambda}
\pi_{\lambda}^{-1}\tau_{X^{\lambda}}$,
where $\pi_{\lambda}:X \to X^{\lambda}$ are the 
induced morphisms.
\end{enumerate}
\end{Prop}
\begin{proof}
\begin{enumerate}
\item This follows from the above lemma.
\item First, there is a natural morphism
$\tau_{X^{\lambda}} \to \pi_{\lambda *}\tau_{X}$.
Taking the adjoint, we obtain 
$f_{\lambda}:\pi^{-1}_{\lambda}\tau_{X^{\lambda}}
\to \tau_{X}$.
Taking the limit, we obtain 
$\underrightarrow{\lim}_{\lambda}\pi_{\lambda}^{-1}
\tau_{X^{\lambda}} \to \tau_{X}$.

Next, we show that if $\varphi:\tau_{X} \to \tau_{X}$
is an endomorphism with $\varphi f_{\lambda}=f_{\lambda}$
for any $\lambda$,
then $\varphi$ is the identity.
But this follows from the fact that the projective system is filtered.
We can also construct a natural morphism
$\tau_{X} \to \underrightarrow{\lim}\pi_{\lambda}^{-1}
\tau_{X^{\lambda}}$, using the filteredness.
Combining all these, we obtain the required isomorphism.
\end{enumerate}
\end{proof}

\begin{Prop}
\label{prop:cocomp:patch}
\begin{enumerate}
\item The category $\cat{$\scr{A}$-Sch}$
is small co-complete.
\item 
The category $\cat{$\scr{A}$-Sch}$
admits finite patchings via quasi-compact opens.
Moreover, the inclusion functor
$I:\cat{$\scr{Q}$-Sch} \to \cat{$\scr{A}$-Sch}$
preserves finite patchings via quasi-compact opens.
\end{enumerate}
\end{Prop}
\begin{proof}
\begin{enumerate}
\item Let $\{X_{\lambda}\}$ be a small inductive system
of $\scr{A}$-schemes.
First, we will construct the $\scr{A}$-scheme $X=(|X|,\scr{O}_{X},\beta_{X})$.
The underlying space $|X|$ is given by the colimit of the
underlying spaces $|X_{\lambda}|$.
Set $|\iota_{\lambda}|:|X| \to |X_{\lambda}|$ be the associated morphisms.
The structure sheaf $\scr{O}_{X}$ is defined by the limit
of $|\iota_{\lambda}|_{*}\scr{O}_{X_{\lambda}}$, as a
$\cat{$\sigma $-alg}$-valued sheaf on $X$.

We have a morphism
\[
\alpha_{1}|\iota_{\lambda}|_{*} \scr{O}_{X_{\lambda}}
\simeq |\iota_{\lambda}|_{*}\alpha_{1}\scr{O}_{X_{\lambda}}
\stackrel{|\iota_{\lambda}|_{*}\beta_{X_{\lambda}}}{\longrightarrow}
|\iota_{\lambda}|_{*}\tau_{X_{\lambda}},
\]
which extends to give
$\alpha_{1}\scr{O}_{X} \to |\iota_{\lambda}|_{*}\tau_{X_{\lambda}}$.
Taking the limit and using Proposition \ref{prop:limit:tau}, we obtain
\[
\beta_{X}:\alpha_{1}\scr{O}_{X} \to \underleftarrow{\lim}_{\lambda}
|\iota_{\lambda}|_{*}\tau_{X_{\lambda}}
\simeq \tau_{X}.
\]
We will verify that restriction maps reflect localizations.
Let $\scr{O}_{X}(U) \to \scr{O}_{X}(V)$
be a restriction map, and let $Z=U\setminus V$
be the closed subset of $U$.
Let $Z_{\lambda}$ be the inverse image of
$Z$ by $X_{\lambda} \to X$,
so that we will denote $Z$ by $(Z_{\lambda})_{\lambda}$.
It suffices to show that
if $a=(a_{\lambda})_{\lambda} \in \scr{O}_{X}(U)$
satisfies $\beta_{X}\alpha_{2}(a) \geq Z$,
then $a$ is invertible in $\scr{O}_{X}(V)$.
Since $\beta_{X_{\lambda}}\alpha_{2}(a_{\lambda}) \geq Z_{\lambda}$,
$a_{\lambda}$ is invertible in $\scr{O}_{X_{\lambda}}(V_{\lambda})$,
where $V=(V_{\lambda})_{\lambda}$.
The uniqueness of the inverse element shows that
$a$ is also invertible in $\scr{O}_{X}(V)$.
Thus, we have defined an $\scr{A}$-scheme $X$.
We also have natural morphisms $\iota_{\lambda}:X_{\lambda} \to X$.

We will show that $X$ is actually the colimit.
Let $j_{\lambda}:X_{\lambda} \to Y$ be morphisms,
compatible with the transition morphisms.
There is a unique natural morphism $|j|:|X| \to |Y|$
between the underlying spaces.
The morphisms $\scr{O}_{Y} \to |j_{\lambda}|_{*}\scr{O}_{X_{\lambda}}$
give
\[
j^{\#}:\scr{O}_{Y} \to \underleftarrow{\lim}_{\lambda}
|j_{\lambda}|_{*}
\scr{O}_{X_{\lambda}}
\simeq \underleftarrow{\lim}_{\lambda}|j|_{*}|\iota_{\lambda}|_{*}
\scr{O}_{X_{\lambda}}
\simeq |j|_{*}\underleftarrow{\lim}_{\lambda}|\iota_{\lambda}|_{*}
\scr{O}_{X_{\lambda}}
\simeq |j|_{*}\scr{O}_{X}.
\]
This is the unique morphism which satisfies 
$|j|_{*}\iota_{\lambda}^{\#}\circ j^{\#} =j_{\lambda}^{\#}$.
It is easy to see that $j=(|j|,j^{\#})$
commutes with the support morphisms $\beta_{X}$ and $\beta_{Y}$.
Thus, we obtained a unique morphism $j:X \to Y$
of $\scr{A}$-schemes, hence $X$ is indeed the co-limit.

\item We only have to show that
if $X$ is obtained by patching $X_{1},\cdots,X_{n}$
by quasi-compact opens, then $\{X_{i} \to X\}$
is a covering. By induction,
it suffices to prove for $n=2$.
There is a surjective morphism $X_{1} \amalg X_{2} \to X$,
hence it remains to show that:

If $R_{1}$ and $R_{2}$ are two II-rings,
then any element of the spectrum of $R=R_{1} \times R_{2}$
are in the image of $\Spec R_{1} \to \Spec R$
or that of $\Spec R_{2} \to \Spec R$.

This is easy to prove, so we will skip.

It is clear from the construction
that the functor $I$ preserves finite patching via
quasi-compact opens.
\end{enumerate}
\end{proof}

\begin{Rmk}
Even though small coproducts exist in the category
of $\scr{A}$-schemes, their behavior is somewhat
different from those in schemes.
For example, a point of an infinite coproduct $X=\amalg_{\lambda} X_{\lambda}$
does not necessarily come from a point of some $X_{\lambda}$,
i.e strictly speaking, $\{X_{\lambda} \to X\}$
is not a covering of $X$.
\end{Rmk}

\begin{Exam}
The spectrum functor
is mal-behaved, when we consider
infinite product of rings:
the underlying space $\Spec \prod_{n} R_{n}$
does not coincide with the co-product $\amalg_{n} \Spec R_{n}$,
even in the category of coherent spaces.

Here is a typical counterexample:
Let $k$ be a field, and set
$R=\prod_{n \in \NNN}R_{n}$,
where $R_{n}=k[x]/(x^{n})$.
Then, the spectrum of $R_{n}$ is a point
for any $R_{n}$, hence $\Spec \prod\alpha_{1}R_{n}$
must be the set of all ultrafilters over $\NNN$,
in particular, its Krull dimension is zero.

On the other hand,
the Krull dimension of $\Spec R$ is not zero:
fix a non-principal ultrafilter $\mathfrak{U}$ on $\NNN$,
and define an ideal $\mathfrak{M}$ of $R$ as
\[
f=(f_{n})_{n} \in \mathfrak{M}
\Leftrightarrow f_{n} \notin R_{n}^{\times} \quad\rom{a.e. $\mathfrak{U}$}.
\] 
Here, $P(s)$ a.e $\mathfrak{U}$ for a condition
$P(s)$ of $s$ means that
the set $\{s \mid P(s)\}$ belongs to $\mathfrak{U}$.
This is a maximal ideal of $R$
(in fact, any maximal ideal of $R$ is of this form).
On the other hand,
define an ideal $\mathfrak{p}$ of $R$ as
\[
f=(f_{n})_{n} \in \mathfrak{p}
\Leftrightarrow \rom{For any $c>0$, $f_{n} \in (x^{\lceil cn \rceil})$
a.e. $\mathfrak{U}$}.
\]
This is also a prime ideal, and obviously
smaller than $\mathfrak{M}$.
Hence, the Krull dimension of $R$ is not $0$.
In fact, one can prove similarly
that the Krull dimension of $R$ is infinite.

It is obvious that $\{\Spec R_{n} \to \Spec R\}_{n}$
is not a covering of $\Spec R$.
\end{Exam}

\subsection{Decomposition of morphisms}
In this subsection,
we prove that there is a functorial
decomposition of morphisms in the category
of $\scr{A}$-schemes. This
decomposition plays an important role
in the proof of completeness of $\cat{$\scr{A}$-Sch}$,
since it gives an upper bound of the cardinality
of a family of morphisms.
Also, note that this decomposition
is peculiar to the category of $\scr{A}$-schemes.
\begin{Def}
A morphism $f:X \to Y$ of $\scr{A}$-schemes
is a \textit{P-morphism} if:
\begin{enumerate}
\item $|f|$ is epic, i.e. $f^{-1}:C(Y)_{\cpt} \to C(X)_{\cpt}$ is injective,
\item $f^{\#}:\scr{O}_{Y} \to f_{*}\scr{O}_{X}$
is injective.
\end{enumerate}
\end{Def}
Let us mention some trivial facts:
\begin{Prop}
\begin{enumerate}
\item A P-morphism is epic.
\item P-morphisms are stable under compositions.
\item If $gf$ is a P-morphism, then so is $g$.
\end{enumerate}
\end{Prop}
These are all obvious,
so we will skip the proof.

\begin{Prop}
\label{prop:surj:colimit}
Let $\{f_{\lambda}:X_{\lambda} \to Y_{\lambda}\}_{\lambda}$
be an inductive system of P-morphisms of $\scr{A}$-schemes,
and set $X_{\infty}=\underrightarrow{\lim}_{\lambda}X_{\lambda}$
and $Y_{\infty}=\underrightarrow{\lim}_{\lambda} Y_{\lambda}$.
Then, the natural morphism $f:X_{\infty} \to Y_{\infty}$
is also a P-morphism. 
\end{Prop}
\begin{proof}
First, we will see that $C(Y_{\infty})_{\cpt} \to C(X_{\infty})_{\cpt}$
is injective. Let $Z=(Z_{\lambda})_{\lambda}$
and $W=(W_{\lambda})_{\lambda}$ be two elements
of $C(Y_{\infty})_{\cpt}$ with $f^{-1}(Z)=f^{-1}(W)$.
Since $f^{-1}$ is defined by $(Z_{\lambda})_{\lambda}
\mapsto (f_{\lambda}^{-1}Z_{\lambda})_{\lambda}$,
this implies that $f_{\lambda}^{-1}Z_{\lambda}
=f_{\lambda}^{-1}W_{\lambda}$.
Since $f_{\lambda}$ is a P-morphism,
$Z_{\lambda}$ and $W_{\lambda}$ must coincide for all $\lambda$,
which is equivalent to $Z=W$.
Hence $C(Y_{\infty})_{\cpt} \to C(X_{\infty})_{\cpt}$ is injective.
A similar argument shows that
$\scr{O}_{Y_{\infty}} \to f_{*}\scr{O}_{X_{\infty}}$ is also injective,
so that $f$ is a P-morphism.
\end{proof}

\begin{Def}
\begin{enumerate}
\item
Fix a small index category $I$.
Let $Y^{\bullet}:I \to \cat{$\scr{A}$-Sch}$
be a small projective system of $\scr{A}$-schemes,
and $f:\Delta(X) \to Y^{\bullet}$ be a morphism
in $\cat{$\scr{A}$-Sch}^{I}$,
where $\Delta: \cat{$\scr{A}$-Sch} \to \cat{$\scr{A}$-Sch}^{I}$
is the diagonal functor.
(In the sequel, we simply denote $\Delta(X)$
by $X$ for brevity.)
Let $\scr{S}$ be the set of isomorphism
classes of the commutative diagram
\[
\xymatrix{
X \ar[r]^{f} \ar[d]_{g} & Y^{\bullet} \\
W \ar[ur] 
}
\]
where $g$ is a P-morphism.
$\scr{S}$ is small, from the property of P-morphisms.
Set $I(X,Y^{\bullet})=\underrightarrow{\lim}_{W \in \scr{S}} W$.
Then, by Proposition \ref{prop:surj:colimit},
$X \to I(X,Y^{\bullet})$ becomes a P-morphism.
\item Let $f:X \to Y^{\bullet}$ be as above.
$f$ is a \textit{Q-morphism},
if $X \to I(X,Y^{\bullet})$ is an isomorphism.
\end{enumerate}
\end{Def}

Roughly speaking, a P-morphism 
can be regarded as a schematic surjection
and a Q-morphism as a schematic immersion.
Thus, if $X \to I(X,Y) \to Y$ is the PQ-decomposition
of a morphism $f:X \to Y$,
$I(X,Y)$ can be regarded as the `image scheme' of $f$.

The next proposition is purely category-theoretical.
\begin{Prop}
\label{prop:decomp:exist}
\begin{enumerate}
\item Let $f:X \to Y^{\bullet}$ be a morphism
from an $\scr{A}$-scheme $X$ to
a projective system $Y^{\bullet}$ of $\scr{A}$-schemes.
Then, the morphism $h:I(X,Y^{\bullet}) \to Y^{\bullet}$
is a Q-morphism.
\item A morphism $f:X \to Y$ of $\scr{A}$-schemes
is an isomorphism if and only if $f$
is a P-morphism and Q-morphism.
\end{enumerate}
\end{Prop}
\begin{proof}
\begin{enumerate}
\item Let $g:X \to I(X,Y^{\bullet})$ be the
induced P-morphism.
Set $W=I(I(X,Y^{\bullet}),Y^{\bullet})$, 
and let $\tilde{h}:W \to Y^{\bullet}$
be the induced morphism.
Since $\pi:I(X,Y) \to W$
is a P-morphism, $X \to W$ is also a P-morphism
by (1). Hence, there is a morphism
$\iota:W \to I(X,Y^{\bullet})$
such that 
$\iota \circ \pi \circ g=g$.
This implies that $\iota \circ \pi$ is the identity,
since $g$ is epic.
Hence, $\pi \circ \iota \circ \pi=\pi$,
and this shows that $\pi \circ \iota$ is the
identity since $\pi$ is also epic.
This shows that $I(X,Y^{\bullet}) \to W$
is an isomorphism.
\item The proof is similar to (1).
\end{enumerate}
\end{proof}
\begin{Cor}
\label{cor:decomp:unique}
\begin{enumerate}
\item Let $X \stackrel{f}{\to} Y \stackrel{g}{\to} Z$
be a series of morphisms of $\scr{A}$-schemes.
If $gf$ is a Q-morphism, then so is $f$.
\item
Let $f:X \to Y^{\bullet}$ be a morphism
of $\scr{A}$-schemes.
Then, the decomposition $f=hg$
of $f$ into a P-morphism $g$ and a Q-morphism $h$
is unique up to unique isomorphism.
\end{enumerate}
\end{Cor}
We refer to the decomposition of (2)
as \textit{PQ-decomposition}.
\begin{proof}
\begin{enumerate}
\item Let $h:X \to I(X,Y)$ and $\tilde{f}:X \to I(X,Z)$
be the induced morphisms.
By universality, we have a morphism
$\pi:I(X,Y) \to I(X,Z)$ such that $\pi h=\tilde{f}$.
$\tilde{f}$ is an isomorphism since $gf$ is a Q-morphism.
Hence, $\tilde{f}^{-1}\pi h$ is the identity.
Also, $h\tilde{f}^{-1} \pi h=h$ and
$h$ epic implies that $h \tilde{f}^{-1}\pi$ is also the identity.
Hence, $h$ is an isomorphism.
\item
Let $X \stackrel{g}{\to} W \stackrel{h}{\to} Y^{\bullet}$
be a decomposition of $f$ into a P-morphism and a Q-morphism.
Then, there is a morphism $\pi: W \to I(X,Y^{\bullet})$
by universality.
Since $X \to I(X,Y^{\bullet})$ is a P-morphism,
$\pi$ is also a P-morphism.
Since $h$ is a Q-morphism,
(1) tells that $\pi$ is a Q-morphism,
hence an isomorphism.
\end{enumerate}
\end{proof}

On the other hand,
it seems to be impossible to prove that
Q-morphisms are stable under compositions,
using only categorical operations.
\begin{Lem}
\label{lem:alpha1:filter:cont}
Let $\scr{A}=(\sigma,\alpha_{1},\alpha_{2},\gamma)$
be the schematizable algebraic type. Then:
\begin{enumerate}
\item $\alpha_{1}$ preserves filtered colimits.
\item $\alpha_{1}$ preserves images:
namely, if $f:A \to B$ is a homomorphism
of $\sigma$-algebras, then
$\alpha_{1}(\Imag f)=\Imag (\alpha_{1}f)$.
\end{enumerate}
\end{Lem}
\begin{proof}
\begin{enumerate}
\item Let $\{R_{\lambda}\}_{\lambda}$ be a filtered inductive
system of $\sigma$-algebras,
and set $R_{\infty}=\underrightarrow{\lim}_{\lambda} R_{\lambda}$.
Then we have a natural homomorphism
$\varphi:\underrightarrow{\lim}_{\lambda} \alpha_{1}R_{\lambda}
\to \alpha_{1}(R_{\infty})$.
We will show that $\varphi$ is bijective.

First, we will prove the surjectivity.
For any $\mathfrak{a} \in \alpha_{1}R_{\infty}$,
$\mathfrak{a}$ can be written as
$\sum_{i}^{n} \alpha_{2}(a_{i})$ for some
$a_{i} \in R_{\infty}$, since $\alpha_{2}(R) \subset \alpha_{1}R$
generates $\alpha_{1}R$.
Since the inductive system is filtered, there
exists $\lambda_{0}$ such that $\{a_{i}\}_{i} \subset R_{\lambda_{0}}$.
Then, $\mathfrak{a}$ is contained in the image of
\[
\alpha_{1}R_{\lambda_{0}} \to
\underrightarrow{\lim}_{\lambda} \alpha_{1}R_{\lambda}
\stackrel{\varphi}{\to} \alpha_{1}R_{\infty},
\]
hence in the image of $\varphi$.

Next, we prove the injectivity.
Suppose $\varphi(\mathfrak{a})=\varphi(\mathfrak{b})$
for some $\mathfrak{a}=\sum_{i}\alpha_{2}(a_{i}),
\mathfrak{b}=\sum_{j} \alpha_{2}(b_{j}) \in 
\underrightarrow{\lim}_{\lambda} \alpha_{1}R_{\lambda}$.
Then, $\mathfrak{a}$ and $\mathfrak{b}$ must
coincide in $\alpha_{1}R_{\lambda_{0}}$ for some $\lambda_{0}$,
which shows that $\mathfrak{a}=\mathfrak{b}$
in $\underrightarrow{\lim}_{\lambda} \alpha_{1}R_{\lambda}$.
\item
Let $\mathfrak{a}$ be an element of $\alpha_{1}(\Imag f)$.
Then, 
\[
\mathfrak{a}=\sum_{i} \alpha_{2}f(a_{i})
=(\alpha_{1}f)\sum_{i} \alpha_{2}(a_{i})
\]
for some $a_{i} \in A$.
This shows that $\alpha_{1}(\Imag f) \subset \Imag (\alpha_{1}f)$.
The converse is similar.
\end{enumerate}
\end{proof}

\begin{Prop}
\label{prop:exist:decomp:surj}
Let $f:X \to Y$ be a morphism of $\scr{A}$-schemes.
Then, there exists a decomposition
$X \stackrel{g}{\to} W \stackrel{h}{\to} Y$ of $f$, where
$g$ is a P-morphism, and $h$ satisfies
\begin{enumerate}[(i)]
\item $h^{-1}:C(Y)_{\cpt} \to C(W)_{\cpt}$ is surjective, and
\item $h^{\#}:\scr{O}_{Y} \to h_{*}\scr{O}_{W}$ is 
\textit{stalkwise surjective}, that is,
$\scr{O}_{Y,h(w)} \to \scr{O}_{W,w}$ is surjective for
any $w \in W$.
\end{enumerate}
\end{Prop}
\begin{proof}
Let $R$ be the image of $f^{-1}:C(Y)_{\cpt} \to C(X)_{\cpt}$,
and set $|W|=\Spec R^{\dagger}$.
The structure sheaf $\scr{O}_{W}:R \simeq C(W)_{\cpt}
\to \cat{$\sigma $-alg}$ is defined by the sheafification of 
\[
R \ni Z \mapsto \Imag[ f^{-1}\scr{O}_{Y}(Z) \to \scr{O}_{X}(Z)]_{S},
\]
where $S=\{ a \mid \text{$\beta_{X}\alpha_{2}(a)=1$ in $R_{Z}$} \}$
is a multiplicative system of
$\Imag[ f^{-1}\scr{O}_{Y}(Z) \to \scr{O}_{X}(Z)]$.

The support morphism $\beta_{W}:\alpha_{1}\scr{O}_{W} \to
\tau_{W}$ is defined as follows:
for any $Z \in C(W)_{\cpt}$,
$\alpha_{1}\scr{O}_{W}(Z)$ is locally isomorphic to
\begin{eqnarray*}
\alpha_{1}\Imag [f^{-1}\scr{O}_{Y}(Z) \to \scr{O}_{X}(Z)]_{S}
\simeq \alpha_{1}\underrightarrow{\lim}_{f^{-1}V=Z}
\Imag [\scr{O}_{Y}(V) \to \scr{O}_{X}(Z)]_{S} \\
\simeq \underrightarrow{\lim}_{f^{-1}V=Z}
\Imag [\alpha_{1}\scr{O}_{Y}(V) \to \alpha_{1}\scr{O}_{X}(Z)]_{S}
\end{eqnarray*}
by Lemma \ref{lem:alpha1:filter:cont}.
Since we have a commutative square
\[
\xymatrix{
\alpha_{1}\scr{O}_{Y}(V) \ar[r] \ar[d]_{\beta_{Y}}
 & \alpha_{1}\scr{O}_{X}(Z) \ar[d]^{\beta_{X}} \\
(C(Y)_{\cpt})_{V} \ar[r] & (C(X)_{\cpt})_{Z}
}
\]
and the lower horizontal arrow factors
through $R_{Z}=\tau_{W}(Z)$,
we obtain a homomorphism
$\alpha_{1}\scr{O}_{W}(Z) \to \tau_{W}(Z)$.
Note that the localization by $S$ does not affect.
It is obvious that the restrictions
reflect localizations, hence
$W=(|W|,\scr{O}_{W},\beta_{W})$ is well defined
as an $\scr{A}$-scheme.
$g:X \to W$ is defined by the injections
$g^{-1}:R \to C(X)_{\cpt}$ and
\[
g^{\#}:\scr{O}_{W}(Z)
=\Imag [f^{-1}\scr{O}_{Y}(Z) \to \scr{O}_{X}(Z)]_{S} \to \scr{O}_{X}(Z).
\]
It is obvious that $g$ is a P-morphism.
$h:W \to Y$ is defined by
$h^{-1}:C(Y)_{\cpt} \to R$ and
\[
h^{\#}:
\scr{O}_{Y}(Z) \to
\scr{O}_{W}(g^{-1}Z)
=\Imag[\scr{O}_{Y}(Z) \to \scr{O}_{X}(f^{-1}Z)]_{S}.
\]
Let us verify that $h^{\#}$ is stalkwise surjective, namely,
$\scr{O}_{Y,h(w)} \to \scr{O}_{W,w}$
is surjective for any $w \in W$.
Let $\langle U,a \rangle$ be any element
of $\scr{O}_{W,w}$.
Since $h^{-1}$ is surjective, $U=h^{-1}V$
for some quasi-compact open $V \subset Y$.
The germ $a$ can be expressed as $b/c$,
where $\beta_{X}\alpha_{2}(c)=1$,
and $b,c$ is in the image of
$\scr{O}_{Y}(V) \to \scr{O}_{W}(U)$.
This implies that $c$ is a unit in $\scr{O}_{W,w}$,
hence also a unit in $\scr{O}_{Y,f(w)}$.
Hence $\scr{O}_{Y,h(w)} \to \scr{O}_{W,w}$
is surjective.
\end{proof}

\begin{Cor}
\begin{enumerate}
\item
Let $f:X \to Y$ be a morphism of $\scr{A}$-schemes.
Then, the $W$ constructed in \ref{prop:exist:decomp:surj}
is naturally isomorphic to $I(X,Y)$.
In particular, the followings are equivalent:
\begin{enumerate}[(i)]
\item $f$ is a Q-morphism.
\item $f^{-1}:C(Y)_{\cpt} \to C(X)_{\cpt}$ is surjective,
and $f^{\#}:\scr{O}_{Y} \to h_{*}\scr{O}_{X}$ is 
stalkwise surjective.
\end{enumerate}
\item Q-morphisms are stable under compositions.
\item Q-morphisms are monic.
\end{enumerate}
\end{Cor}
\begin{proof}
\begin{enumerate}
\item
By Proposition \ref{prop:exist:decomp:surj},
there exists a decomposition
$X \stackrel{g}{\to} W \stackrel{h}{\to} Y$ of $f$, where
$g$ is a P-morphism, 
$h^{-1}:C(Y)_{\cpt} \to C(W)_{\cpt}$ is surjective,
and $h^{\#}:\scr{O}_{Y} \to h_{*}\scr{O}_{W}$ is 
stalkwise surjective.
By the universal property,
there is a P-morphism $u:W \to I(X,Y)$.
Note that $I(X,Y)$ is isomorphic to $X$,
since $f$ is a Q-morphism.
Since $h^{-1}$ is surjective,
$u^{-1}:C(X)_{\cpt} \to C(W)_{\cpt}$
is an isomorphism.
Also, the stalkwise surjectivity of $h^{\#}$ implies
that $u^{\#}$ is also an isomorphism.
\item It is clear from (1).
\item Obvious.
\end{enumerate}
\end{proof}

\begin{Lem}
\label{lem:decomp:comp}
If $X \to Y \to Z$ is a series of morphisms
of $\scr{A}$-schemes, then we have a natural isomorphism
\[
I(X,Z) \simeq I(I(X,Y),I(Y,Z)).
\]
\end{Lem}
\begin{proof}
Since $W=I(I(X,Y),I(Y,Z)) \to I(Y,Z)$ and
$I(Y,Z) \to Z$ are Q-morphisms,
the composition $W \to Z$ is also a Q-morphism.
Hence, $X \to W \to Z$ is the
PQ-decomposition of $X \to Z$.
\end{proof}

\begin{Prop}
The PQ-decomposition
is functorial.
\end{Prop}
\begin{proof}
Suppose given a commutative square
\[
\xymatrix{
X_{1} \ar[r] \ar[d] & Y_{1} \ar[d] \\
X_{2} \ar[r] & Y_{2}
}
\]
Then, we will show that there is a natural morphism
$I(X_{1},Y_{1}) \to I(X_{2},Y_{2})$.
By the universality, we have
a unique morphism $I(X_{1},Y_{1}) \to I(X_{1},Y_{2})$.
On the other hand,
we have again a unique morphism $I(X_{1},Y_{2}) \to I(X_{2},Y_{2})$
by Lemma \ref{lem:decomp:comp},
hence combining them gives the required morphism
$I(X_{1},Y_{1}) \to I(X_{2},Y_{2})$.
\end{proof}
Let us summarize what we have obtained
in this subsection:
\begin{Thm}
\label{thm:PQ:decomp}
For any morphism $f:X \to Y$ of $\scr{A}$-schemes,
we have a functorial decomposition
$X \to I(X,Y) \to Y$ of $f$,
where
\begin{enumerate}
\item $X \to I(X,Y)$ is a P-morphism
(in particular, epic), and
\item $I(X,Y) \to Y$ is a Q-morphism
(in particular, monic).
\end{enumerate}
Moreover, the decomposition of the given morphism
$f$ into a P-morphism and a Q-morphism is unique
up to unique isomorphism.
Also, this decomposition is universal:
if $f$ is factors as $X \to W \to Y$
where $X \to W$ is a P-morphism (resp. $W \to Y$
is a Q-morphism), then there is a unique
morphism $W \to I(X,Y)$ (resp. $I(X,Y) \to W$)
making the whole diagram
commutative.
\end{Thm}

\begin{Rmk}
We know that this decomposition
is impossible in the category of schemes.
For example, let $k$ be a field, $X=\Spec k[x, y/x]$,
$Y=\Spec k[x,y]$ where $x$ and $y$ are indeterminants.
The image of the natural morphism $f:X \to Y$
cannot be a scheme: the origin has no
affine neighborhood in the image.
\end{Rmk}

\subsection{Completeness}

\begin{Prop}
The category $\cat{$\scr{Q}$-Sch}$
is finite complete, i.e. there are fiber products.
\end{Prop}
\begin{proof}
The construction of fiber products
are similar to that of general schemes.
Note that we use the fact that
quasi-compact open immersions
are stable under base changes by
quasi-compact morphisms.
\end{proof}

\begin{Rmk}
Let $\sigma$ be the algebraic system of rings.
Then, the natural inclusion
functor $\cat{$\scr{Q}$-Sch} \to \cat{Sch}$
preserves fiber products.
This is clear from the construction.
\end{Rmk}

We already know that
the category of ordinary schemes is not complete.
However:
\begin{Prop}
\label{prop:asch:comp}
The category $\cat{$\scr{A}$-Sch}$
is small complete.
\end{Prop}
\begin{proof}
Let $X^{\bullet}$
be a small projective system of $\scr{A}$-schemes.
Let $\scr{S}$ be the set of all isomorphism classes
of Q-morphisms $Y \to X^{\bullet}$.
$\scr{S}$ is indeed a small set: if $Y \to X^{\bullet}$
is a Q-morphism, then the underlying space of $Y$
and its structure sheaf is generated by those
of $X^{\bullet}$, hence $\scr{S}$ is small.

Let $X$ be the co-limit of $\scr{S}$.
Then, $X$ is the limit of $X^{\bullet}$.
\end{proof}

\begin{Cor}
\label{cor:preserve:fiber:prod}
The natural inclusion functor
$\cat{$\scr{Q}$-Sch} \to \cat{$\scr{A}$-Sch}$
preserves fiber products.
\end{Cor}
\begin{proof}
Let $X, Y$ be $\scr{Q}$-schemes over a $\scr{Q}$-scheme $S$.
We will show that the fiber product $V=X \times_{S} Y$ 
in the category of $\scr{Q}$-schemes is indeed
that in the category of $\scr{A}$-schemes.
\begin{itemize}
\item[Step 1:] If $X,Y,S$ are all affine,
then $V$ is the fiber product in $\cat{$\scr{A}$-Sch}$,
by the adjunction
$\Spec^{\scr{A}}:\cat{$\sigma $-alg} \leftrightarrows
\cat{$\scr{A}$-Sch}: \Gamma$.
\item[Step 2:] Suppose $Y, S$ is affine.
Let $X=\cup_{i} X_{i}$ be an open affine cover of $X$.
Suppose given the following commutative square:
\[
\xymatrix{
Z \ar[r]^{f} \ar[d]_{g} & X \ar[d] \\
Y \ar[r] & S
}
\]
Then, there is a unique morphism $f^{-1}(X_{i}) \to X_{i} \times_{S} Y$
by Step 1 for each $i$, which patches up to give the morphism
$Z \to X \times_{S} Y$.
\item[Step 3:] Same arguments as in Step 2
shows that if $S$ is affine, then $V$ is the fiber product
in $\cat{$\scr{A}$-Sch}$.
\item[Step 4:] Suppose $S=\cup_{i} S_{i}$
is an open affine cover of $S$.
Set $X_{i}=X \times_{S} S_{i}$,
$Y_{i}=Y \times_{S} S_{i}$.
Then, $V_{i}=X_{i} \times_{S_{i}} Y_{i}$
is also the fiber product in $\cat{$\scr{A}$-Sch}$ by Step 3.
Note that $V=\cup_{i} V_{i}$ is an quasi-compact open
covering. Suppose given a commutative diagram as in Step 2,
and set $Z_{i}=f^{-1}(X_{i})$. This coincides with $g^{-1}(Y_{i})$.
Then we have a unique morphism $Z_{i} \to V_{i}$ for each $i$,
which patches up to give a morphism $Z \to V$.
\end{itemize}
\end{proof}
The next proposition is purely category-theoretical.
\begin{Prop}
\begin{enumerate}
\item Q-morphisms are stable under pullbacks.
\item P-morphisms are stable under pushouts.
\end{enumerate}
\end{Prop}
\begin{proof}
We will only prove (1):
the proof of (2) can be proven by the dual argument.
Consider the following pullback diagram:
\[
\xymatrix{
X_{T} \ar[r]^{\tilde{f}} \ar[d]_{\tilde{g}} & X \ar[d]^{g} \\
T \ar[r]_{f} & S
}
\]
Suppose $f$ is a Q-morphism.
Then the PQ-decomposition 
\[
X_{T} \stackrel{\alpha(\tilde{f})}{\to}
\Imag\tilde{f}=I(X_{T}, X)  \stackrel{\beta(\tilde{f})}{\to} X
\]
of $\tilde{f}$ gives
a morphism $w:\Imag \tilde{f} \to T$
such that $w\circ \alpha(\tilde{f})=\tilde{g}$
and $f \circ w=g \circ \beta(\tilde{f})$.
By the universal property of the pullback,
there exists a unique morphism $u:\Imag \tilde{f} \to X_{T}$
such that $\tilde{g}\circ u=w$ and $\tilde{f} \circ u=\beta(\tilde{f})$.
The second equality shows that $u$ is a Q-morphism,
hence it suffices to show that $u$ is a P-morphism.
To prove this, we will show that $u \circ \alpha(\tilde{f})$
is the identity. Since
\[
\tilde{g} \circ u \circ \alpha(\tilde{f})=w \circ \alpha(\tilde{f})
=\tilde{g}, \quad \text{and} \quad
\tilde{f}\circ u \circ \alpha(\tilde{f})
=\beta(\tilde{f}) \circ \alpha(\tilde{f})=\tilde{f},
\]
the universal property of the pullback shows that $u \circ \alpha(\tilde{f})$
is the identity.
\end{proof}

\subsection{Filtered limits}

\begin{Prop}
\label{prop:filter:limit}
Let $X^{\bullet}=\{X^{\lambda}\}$ be a small filtered projective system
of $\scr{A}$-schemes, and $Y$ be the limit
of $X^{\bullet}$.
Then, the underlying space of $Y$
coincides with the limit $\underleftarrow{\lim} X^{\lambda}$
in the category of coherent spaces.
The structure sheaf $\scr{O}_{Y}$
coincides with the colimit 
$\underrightarrow{\lim} p_{\lambda}^{-1}\scr{O}_{X^{\lambda}}$,
where $p_{\lambda}:Y \to X^{\lambda}$
are the natural morphisms.
\end{Prop}
\begin{proof}
Let $X^{\infty}$ be the limit of the
$X^{\bullet}$ in the category of coherent spaces,
and set $\scr{O}_{X^{\infty}}=
\underrightarrow{\lim}\pi_{\lambda}^{-1}\scr{O}_{X^{\lambda}}$,
where $\pi_{\lambda}:X^{\infty} \to X^{\lambda}$
is the natural morphism of coherent spaces.
First, we will construct the support morphism $\beta_{X^{\infty}}$.
By Proposition \ref{prop:limit:tau}, we have a natural map
\[
\alpha_{1}\pi_{\lambda}^{-1}\scr{O}_{X^{\lambda}}
\to \pi^{-1}_{\lambda}\tau_{X^{\lambda}}
\to \underrightarrow{\lim}_{\lambda} 
\pi^{-1}_{\lambda}\tau_{X^{\lambda}}
 \simeq \tau_{X^{\infty}}.
\]
This gives a natural map
$\underrightarrow{\lim}_{\lambda}
(\alpha_{1}\pi^{-1}_{\lambda}\scr{O}_{X^{\lambda}})
\to \tau_{X^{\infty}}$.
The left-hand side is isomorphic to $\alpha_{1}\scr{O}_{X^{\infty}}$
since $\alpha_{1}$ is filtered co-continuous
by Lemma \ref{lem:alpha1:filter:cont}.
Therefore, we obtain the required
morphism 
$\beta_{X^{\infty}}:\alpha_{1}\scr{O}_{X^{\infty}}
\to \tau_{X^{\infty}}$.
It is obvious that restrictions reflects localizations.
Hence, we have constructed an $\scr{A}$-scheme
$X^{\infty}=(X^{\infty},\scr{O}_{X^{\infty}},\beta_{X^{\infty}})$.
There are also natural morphisms
$\pi_{\lambda}:X^{\infty} \to X^{\lambda}$
of $\scr{A}$-schemes, compatible with the transitions.

We will show that $X^{\infty}$ is naturally isomorphic to $Y$.
Since we already have a morphism $X^{\infty} \to Y$
by the universal property of $Y$,
it suffices to show that:
\begin{enumerate}[(i)]
\item If $\varphi:X^{\infty} \to X^{\infty}$
is a endomorphism with $\pi_{\lambda}\varphi=\pi_{\lambda}$
for any $\lambda$, then $\varphi$ is the identity.
\item There exists a morphism
$\psi:Y \to X^{\infty}$ with $\pi_{\lambda}\psi=p_{\lambda}$
for any $\lambda$.
\end{enumerate}
First, we prove (i).
It is obvious that $\varphi$ is the identity on the
underlying space. For the structure sheaves, we have the
commutative diagram:
\[
\xymatrix{
\pi^{-1}_{\lambda}\scr{O}_{X^{\lambda}} \ar[d]_{\pi_{\lambda}^{\#}}
\ar[rd]^{\pi_{\lambda}^{\#}} \\
\scr{O}_{X^{\infty}} \ar[r]_{\varphi^{\#}} &
\scr{O}_{X^{\infty}}
}
\]
which shows that $\varphi^{\#}$ is the identity,
since $\scr{O}_{X^{\infty}}=
\underrightarrow{\lim}_{\lambda}\pi^{-1}_{\lambda}\scr{O}_{X^{\lambda}}$.

It remains to prove (ii).
There is a natural morphism $|\psi|:|Y| \to |X^{\infty}|$
between the underlying spaces.
Since there are morphisms
\[
p_{\lambda}^{\#}:\scr{O}_{X^{\lambda}} \to
p_{\lambda *}\scr{O}_{Y} \simeq \pi_{\lambda *}|\psi|_{*}\scr{O}_{Y},
\]
these give morphisms $\pi_{\lambda}^{-1}\scr{O}_{X^{\lambda}} \to
|\psi|_{*}\scr{O}_{Y}$.
It is obvious that these are compatible with the
transition morphisms, hence
we obtain $\scr{O}_{X^{\infty}} \to \psi_{*}\scr{O}_{Y}$.
Also, this morphism commutes with
$\beta_{X^{\infty}}$ and $\beta_{Y}$,
hence we have a morphism of $\scr{A}$-schemes.
It is obvious that $\pi_{\lambda}\psi=p_{\lambda}$.
\end{proof}

\begin{Prop}
\label{prop:imm:filter:limit}
Let $\{g_{\lambda}:X^{\lambda} \to Y^{\lambda}\}_{\lambda}$
be a filtered projective system of Q-morphisms of $\scr{A}$-schemes,
and set $X^{\infty}=\underleftarrow{\lim}_{\lambda}X^{\lambda}$
and $Y^{\infty}=\underleftarrow{\lim}_{\lambda} Y^{\lambda}$.
Then, the natural morphism $g:X^{\infty} \to Y^{\infty}$
is also a Q-morphism. 
\end{Prop}
\begin{proof}
First, we will see that
$g^{-1}:C(Y^{\infty})_{\cpt} \to C(X^{\infty})_{\cpt}$
is surjective.
Since $C(X^{\infty})_{\cpt}=\underrightarrow{\lim}_{\lambda}
C(X^{\lambda})$ is a filtered colimit,
any element $Z$ of $C(X^{\infty})_{\cpt}$
is in the image of $C(X^{\lambda})_{\cpt}$ for some $\lambda$.
Since $C(Y^{\lambda})_{\cpt} \to C(X^{\lambda})_{\cpt}$ is surjective,
there is an element $W \in C(Y^{\lambda})_{\cpt}$
such that $g_{\lambda}^{-1}W=Z$.
Hence, $g^{-1}\pi_{\lambda}^{-1}W=g_{\lambda}^{-1}W=Z$.
This shows that $g^{-1}$ is surjective.
A similar argument shows that
$\scr{O}_{Y^{\infty}} \to g_{*}\scr{O}_{X^{\infty}}$
is also stalkwise surjective.
\end{proof}

\section{Separated and Proper morphisms}

\subsection{Reduced schemes}
In the sequel, the algebraic system
is that of rings.
\begin{Def}
an $\scr{A}$-scheme $X$ is \textit{reduced},
if $\beta_{X}\alpha_{2}(a)=0$ implies
$a=0$ for any section $a \in \scr{O}_{X}$.
\end{Def}
Note that if $X$ is reduced, then
the radical of any ring of sections
become $0$.
\begin{Prop}
\label{prop:red:closed:str}
Let $X$ be an $\scr{A}$-scheme
and $Z$ be a closed subset of the underlying space of $X$.
Then, there is a reduced $\scr{A}$-scheme structure
$(Z,\scr{O}_{Z},\beta_{Z})$ on $Z$,
referred to as the \textit{reduced induced subscheme structure of $Z$}.
Also, there is a Q-morphism $Z \to X$,
satisfying the following universal property:

If $Y \to X$ is a morphism of $\scr{A}$-schemes,
with $Y$ reduced and the set-theoretic image contained in $Z$,
then it factors through $Z$.
\end{Prop}
\begin{proof}
The structure sheaf $\scr{O}_{Z}$ of $Z$ is defined by
the sheafification of the presheaf
\[
W \mapsto \underrightarrow{\lim}_{V+Z \geq W}
\scr{O}_{X}(V)/\{ a\mid \beta_{X}\alpha_{2}(a)\cdot W \leq Z\},
\]
where the colimit runs through all $V \in C(X)_{\cpt}$ in $X$
such that $V+Z \geq W$.
For any closed $W$ in $Z$,
and closed $V$ in $X$ satisfying $V+Z \geq W$,
the morphism $\alpha_{1}\scr{O}_{X}(V) \to \tau_{X}(W)/Z$
induced from $\beta_{X}$
factors through 
$\alpha_{1}\scr{O}_{X}(V)/\{ a\mid \beta_{X}\alpha_{2}(a)\cdot W \leq Z\}$.
Hence, we can define 
the support morphism 
$\beta_{Z}:\alpha_{1}\scr{O}_{Z} \to \tau_{Z}$.
These give the reduced $\scr{A}$-scheme structure $(Z,\scr{O}_{Z},\beta_{Z})$
on $Z$.

We will give a morphism $\iota:Z \to X$ of $\scr{A}$-schemes.
The map between the underlying spaces is obvious.
For any closed $W$ in $X$, we have a natural morphism
$\scr{O}_{X}(W) \to \iota_{*}\scr{O}_{Z}(W)=\scr{O}_{Z}(W+Z)$,
which gives a stalkwise surjective morphism
$\scr{O}_{X} \to \iota_{*}\scr{O}_{Z}$.
It is clear that this gives a Q-morphism.

Suppose we are given a morphism $f:Y \to X$
with $Y$ reduced and $\Imag f \subset Z$.
Then, the morphism $f^{\#}:\scr{O}_{X} \to f_{*}\scr{O}_{Y}$
factors through $\iota_{*}\scr{O}_{Z}$, since $Y$
is reduced. Thus, $f$ factors through $Z$.
\end{proof}

\begin{Thm}
Let $\cat{red.$\scr{A}$-Sch}$ be the full subcategory
of $\scr{A}$-schemes, which consists of reduced $\scr{A}$-schemes.
Then, the underlying functor
$U:\cat{red.$\scr{A}$-Sch} \to \cat{$\scr{A}$-Sch}$
has a right adjoint, and the counit morphism is a Q-morphism.
\end{Thm}
\begin{proof}
Proposition \ref{prop:red:closed:str}
tells that for any $\scr{A}$-scheme $X$,
there exists a Q-morphism $\eta:X^{\red} \to X$
from a reduced $\scr{A}$-scheme $X^{\red}$,
the underlying space of which coincides with $X$.
Any morphism $X \to Y$ of $\scr{A}$-schemes
gives rise to a morphism $X^{\red} \to Y^{\red}$
of reduced $\scr{A}$-schemes, by 
the universal property.
Hence, we have a functor 
$\red:\cat{$\scr{A}$-Sch} \to \cat{red.$\scr{A}$-Sch}$.
We see that this is the right adjoint of $U$.
The unit $\epsilon:\Id \to \red\circ U$
is the identity. The counit $\eta$ is already given.
\end{proof}

\begin{Prop}
Let $f:X \to Y$ be a morphism of $\scr{A}$-schemes,
with $X$ reduced.
Let $X \to I(X,Y) \to Y$
be the PQ-decomposition.
Then, $I(X,Y)$ is also reduced.
\end{Prop}
\begin{proof}
Since $X$ is reduced, the P-morphism $X \to I(X,Y)$ factors
through $I(X,Y)^{\red}$.
Therefore $I(X,Y)^{\red} \to I(X,Y)$ is also a P-morphism,
hence an isomorphism.
\end{proof}

\begin{Def}
an $\scr{A}$-scheme $X$ is \textit{integral},
if $\scr{O}_{X}(Z)$ is integral for any $Z$.
\end{Def}
Note that, in the category of $\scr{A}$-schemes,
integrality is a weaker condition than
`irreducible and reduced'.
\begin{Prop}
\begin{enumerate}
\item Let $X$ be a reduced irreducible $\scr{A}$-scheme,
and $x_{0} \rightsquigarrow x_{1}$ a specialization.
Then, the restriction map $\scr{O}_{X,x_{1}} \to \scr{O}_{X,x_{0}}$
is an injection.
Also,
$\scr{O}_{X,\xi}$ is a field, where $\xi$ is the generic point of $X$.
\item An $\scr{A}$-scheme $X$ is integral
if $X$ is reduced and irreducible.

\item Let $f:X \to Y$ be a dominant morphism
of $\scr{A}$-schemes, with $Y$ reduced.
Then, $f^{\#}:\scr{O}_{Y,f(x)} \to \scr{O}_{X,x}$
is injective for any $x \in X$.
\end{enumerate}
\end{Prop}
\begin{proof}
\begin{enumerate}
\item 
Let $\langle U,a \rangle$ be a germ of $\scr{O}_{X,x_{1}}$
which is in the kernel of $\scr{O}_{X,x_{1}} \to \scr{O}_{X,x_{0}}$.
Then, $a|_{V}=0$ for some neighborhood $V$ of $x_{0}$.
Since $X$ is irreducible, $V$ is dense in $X$,
hence also in $U$.
This implies that $\beta_{X}\alpha_{2}(a)=0$.
Since $X$ is reduced, $a$ must be $0$.
Therefore,
the map $\scr{O}_{X,x_{1}} \to \scr{O}_{X,x_{0}}$
is injective.

Let $a$ be a non-zero element of $\scr{O}_{X,\xi}$.
Then, $a$ is invertible, since $\beta_{X}\alpha_{2}(a) \neq 0$ and
the restriction maps reflect localizations.
\item Let $a,b \in \scr{O}_{X}(Z)$ be two sections
with $ab=0$. Then, 
$\beta_{X}\alpha_{2}(a)\cdot\beta_{X}\alpha_{2}(b)
=\beta_{X}\alpha_{2}(ab)=0$.
Since $X$ is irreducible, we may assume that $\beta_{X}\alpha_{2}(a)=0$.
Since $X$ is reduced, $a$ must be $0$.
\item Suppose $a \in \scr{O}_{Y,f(x)}$
is in the kernel of $f^{\#}$.
Then,
\[
|f^{-1}|\beta_{Y}\alpha_{2}(a)=\beta_{X}f^{\#}\alpha_{2}(a)=0.
\]
Since $f$ is dominant, we have $\beta_{Y}\alpha_{2}(a)=0$.
$Y$ is reduced, hence $a=0$.
\end{enumerate}
\end{proof}

\begin{Exam}
\label{exam:formal:sch}
Let $A$ be a noetherian ring,
and $I \subset A$ be a non-trivial ideal.
We can consider a colimit 
$\mathfrak{X}=\underrightarrow{\lim}_{n} \Spec^{\scr{A}} A/I^{n}$
in the category of $\scr{A}$-schemes.
This becomes an integral $\scr{A}$-scheme
if $\hat{A}=\underleftarrow{\lim}_{n}A/I^{n}$ is a domain.
On the other hand, the underlying space
of $\mathfrak{X}$ coincides with the support of $I$,
hence $\mathfrak{X}$ is not reduced.
In fact, $\mathfrak{X}$ can be regarded as a noetherian formal scheme.
\end{Exam}

\subsection{Right lifting properties}
\begin{Def}
Let $\scr{C}$ be a category,
and $\mathcal{I}$ be a non-empty family
of morphisms in $\scr{C}$.
Fix a morphism $f:X \to Y$ of $\scr{C}$.
Given a morphism $g:A \to B$ in $\mathcal{I}$,
we have a natural map
$\varphi_{f,g}:\Hom_{\scr{C}}(B,X) \to \Hom_{\Mor(\scr{C})}(g,f)$,
where $\Mor(\scr{C})$ is the category of morphisms
in $\scr{C}$.
 We say that $f$ is
\textit{$\mathcal{I}$-separated}
(resp. \textit{$\mathcal{I}$-universally closed,
$\mathcal{I}$-proper})
if $\varphi_{f,g}$ is injective (resp. surjective, bijective)
for any $g \in \mathcal{I}$.
\end{Def}
\begin{Rmk}
The conventional definition of properness
includes the condition `of finite type'.
However, we dropped this condition here,
since it does not seem to be essential
when we discuss about valuative criteria.
Moreover, note that morphisms of finite type are not stable
under taking limits, while the other conditions do.
\end{Rmk}
Here, we list up some properties of
$\mathcal{I}$-separated morphisms, etc.
The proofs are all straightforward.
\begin{Prop}
Let $\scr{C}$ be a category,
and $\mathcal{I}$ be a non-empty family of morphisms in $\scr{C}$.
\begin{enumerate}
\item Isomorphisms are $\mathcal{I}$-proper.
\item Monics are $\mathcal{I}$-separated.
\item The class of $\mathcal{I}$-separated
(resp. $\mathcal{I}$-universally closed, $\mathcal{I}$-proper)
morphisms are stable under compositions.
Thus, we can think of the subcategory $\scr{C}(\mathcal{I})_{i}$
(resp. $\scr{C}(\mathcal{I})_{s}$, $\scr{C}(\mathcal{I})_{b}$)
of $\scr{C}$ consisting of $\mathcal{I}$-separated 
(resp. $\mathcal{I}$-universally closed, $\mathcal{I}$-proper)
morphisms.
\item If $\scr{C}$ has fiber products,
then $\scr{C}(\mathcal{I})_{i}$,
$\scr{C}(\mathcal{I})_{s}$ and $\scr{C}(\mathcal{I})_{b}$
are stable under pull backs.
\item If $\scr{C}$ is small complete,
then so is $\scr{C}(\mathcal{I})_{i}$,
$\scr{C}(\mathcal{I})_{s}$ and $\scr{C}(\mathcal{I})_{b}$.
Also, the inclusion functor $\scr{C}(\mathcal{I})_{*} \to \scr{C}$
is small continuous for $*=i,s,b$.
\item If $gf$ is $\mathcal{I}$-separated,
then $f$ is $\mathcal{I}$-separated.
\item If $gf$ is $\mathcal{I}$-universally closed
(resp. $\mathcal{I}$-proper)
and $g$ is $\mathcal{I}$-separated,
then $f$ is $\mathcal{I}$-universally closed
(resp. $\mathcal{I}$-proper).
\end{enumerate}
\end{Prop}

\begin{Def}
Let $\scr{C}$ be a category,
and $\mathcal{I}$ be a non-empty family of morphisms in $\scr{C}$.
Let $X$ be an object of $\scr{C}$.
\begin{enumerate}
\item
A family $\{U_{\lambda} \to X\}_{\lambda}$ of morphisms 
with target $X$ is an \textit{$\mathcal{I}$-covering} of $X$,
if for any morphism $f:A \to B$ in $\mathcal{I}$
and any morphism $g:B \to X$,
$g$ lifts to $B \to U_{\lambda}$ for some $\lambda$:
\[
\xymatrix{
A \ar[d]_{f} & U_{\lambda} \ar[d] \\
B \ar[r]_{g} \ar@{.>}[ur] & X
}
\]
\item Suppose $\scr{C}$ has fiber products.
Let $\mathcal{J}$ be a family of morphisms in $\scr{C}$.
We say $\mathcal{J}$ is \textit{local on the base}
with respect to $\mathcal{I}$,
if the following holds:

let $f:X \to Y$ be a morphism,
and $\{U_{\lambda} \to Y\}_{\lambda}$ be an $\mathcal{I}$-covering of $Y$.
Set $X_{\lambda}=X \times_{Y} U_{\lambda}$.
Then, $f$ is contained in $\mathcal{J}$
if $f_{\lambda}: X_{\lambda} \to U_{\lambda}$
is contained in $\mathcal{J}$ for any $\lambda$.
\end{enumerate}
\end{Def}
Then, we also have:
\begin{Prop}
Let $\scr{C}$ be a category,
and $\mathcal{I}$ be a non-empty family of morphisms in $\scr{C}$.
Then, $\scr{C}(\mathcal{I})_{i}$
 (resp. $\scr{C}(\mathcal{I})_{s}$, $\scr{C}(\mathcal{I})_{b}$)
is local on the base.
\end{Prop}

\subsection{Separatedness}
Throughout this subsection, we fix a base $\scr{A}$-scheme $S$.
The category of $\scr{A}$-schemes
over $S$ is denoted by $\cat{$\scr{A}$-Sch/$S$}$.
Also, fix a family of morphisms $\mathcal{I}$.
\begin{Def}
Let $\cat{$\mathcal{I}$-sep.$\scr{A}$-Sch/$S$}$ be a full subcategory
of $\cat{$\scr{A}$-Sch/$S$}$ consisting of 
$\scr{A}$-schemes, which is $\mathcal{I}$-separated
over $S$.
\end{Def}
\begin{Prop}
\label{prop:sep:functor}
The underlying functor
\[
U:\cat{$\mathcal{I}$-sep.$\scr{A}$-Sch/$S$} \to \cat{$\scr{A}$-Sch/$S$}
\]
has a left adjoint, and the unit morphism is a P-morphism.
\end{Prop}
\begin{proof}
Let $X$ be an $\scr{A}$-scheme over $S$.
Let $\scr{S}$ be the set of isomorphism classes
of P-morphisms $X \to Y$, where $Y$ is an $\scr{A}$-scheme
which is $\mathcal{I}$-separated over $S$.
We see that $\scr{S}$ is a small set,
since the elements are represented by P-morphisms
with the source fixed.
Suppose given a morphism $f:X \to Z$,
where $Z$ is $\mathcal{I}$-separated over $S$.
The PQ-decomposition $X \to Z^{\prime} \to Z$
gives a P-morphism $X \to Z^{\prime}$,
where $Z^{\prime}$ is $\mathcal{I}$-separated over $S$,
since the Q-morphism $Z^{\prime} \to Z$ is monic.
Therefore, $f$ factors through a morphism
in $\scr{S}$.
Using Freyd's adjoint functor theorem (\cite{CWM}, p121),
we obtain the result.
\end{proof}

\subsection{Valuative criteria}
\begin{Def}
A Q-morphism is a \textit{closed immersion}
if its image is closed.
\end{Def}

\begin{Def}
\label{def:morph:par:spe}
Let $\mathcal{I}$ be a family of morphisms
in the category of $\scr{A}$-schemes.
We say that $\mathcal{I}$ \textit{parametrizes specializations}
if the following conditions hold:
\begin{enumerate}
\item For any morphism $f:U \to V$ in $\mathcal{I}$,
$V$ is irreducible, reduced and local:
the generic point will be denoted by $\xi$,
and the closed point by $\eta$.
$U$ is a one-point reduced $\scr{A}$-scheme
(hence, a spectrum of a field)
with its image onto $\xi$.
\item Let $f:U \to V$ be a morphism in $\mathcal{I}$,
and $g,h:V \to X$ a pair of morphisms with
$gf=hf$ and $g(\eta)=h(\eta)$.
Then, $g=h$. 
\item Let $g:X \to Y$ be a dominant morphism
between two reduced irreducible $\scr{A}$-schemes,
and $y \in Y$.
Then, there exists a morphism $f:U \to V$
in $\mathcal{I}$ and a commutative square
\[
\xymatrix{
U \ar[r] \ar[d]_{f} & X \ar[d]^{g} \\
V \ar[r] & Y
}
\]
such that $U$ maps onto the generic point of $X$,
and $\eta \in V$ maps onto $y$.
\item Let $f:U \to V$ be a morphism in $\mathcal{I}$,
and $Z$ be a reduced irreducible $\scr{A}$-scheme.
If $f$ factors as
$U \stackrel{g}{\to} Z \stackrel{h}{\to} V$ where
$g:U \to Z$ is dominant, and $h:Z \to V$ is surjective
on the underlying space,
then there is a section $V \to Z$ of $h$.
\end{enumerate}
\end{Def}

\begin{Prop}
\label{prop:val:crit:sep}
Suppose $\mathcal{I}$ is a family of morphisms,
parametrizing specializations.
Let $X$ be an $\scr{A}$-schemes over $S$.
Then, the followings are equivalent:
\begin{enumerate}[(i)]
\item $X$ is separated over $S$, i.e.
the diagonal morphism $\Delta:X \to X \times_{S} X$
is a closed immersion.
\item $X$ is $\mathcal{I}$-separated over $S$.
\end{enumerate}
\end{Prop}
Note that the diagonal morphism is monic,
since it is the equalizer of 
$\pi_{1},\pi_{2}:X \times_{S} X \rightrightarrows X$,
where $\pi_{i}$ is the $i$-th projection for $i=1,2$.
\begin{proof}
(i)$\Rightarrow$(ii):
Suppose there is a commutative diagram
\[
\xymatrix{
U \ar[r] \ar[d] & X \ar[d] \\
V \ar[r] \ar@<.5ex>[ur]^{f} \ar@<-.5ex>[ur]_{g} & S
}
\]
with $U \to V$ a morphism in $\mathcal{I}$.
Then we obtain a commutative diagram
\[
\xymatrix{
U \ar[r] \ar[d] & X \ar[d]^{\Delta}\\
V \ar[r]_{(f,g)\quad} & X \times_{S} X
}
\]
Since $\Delta$ is a closed immersion and 
$(f,g)(\xi) \rightsquigarrow (f,g)(\eta)$ is a specialization,
$(f,g)(\eta)$ is in the image of $\Delta$.
This shows that $\pi_{1}\circ (f,g)=\pi_{2} \circ (f,g)$
from condition (2) of \ref{def:morph:par:spe}.
Hence, $(f,g)$ lifts to give a morphism $h:V \to X$
since $\Delta:X \to X \times_{S} X$ is the equalizer
of $\pi_{1},\pi_{2}$,
and $h$ coincides with $f=\pi_{1} \circ (f,g)$
and $g=\pi_{2} \circ (f,g)$.
Therefore, $f$ and $g$ must coincide.

(ii)$\Rightarrow$(i):
It suffices to show that the image of $\Delta$
is stable under specializations, by 
Corollary \ref{cor:image:close:spe}.

So let $\Delta(x)\rightsquigarrow y$ be a specialization,
and $Z=\overline{\{x\}}$ be the closed subset of $X$,
with the reduced induced subscheme structure.
Also, let $W$ be the closure of $\Delta(Z)$ in $Y$,
with the reduced induced subscheme structure.
Then, by condition (3) of \ref{def:morph:par:spe},
we have a commutative diagram
\[
\xymatrix{
U \ar[r] \ar[d] & Z \ar[r] \ar[d] & X \ar[d]^{\Delta} \\
V \ar[r] & W \ar[r] & X \times_{S} X
}
\]
where $U \to V$ is a morphism in $\mathcal{I}$.
Let $u$ be the image of $U \to X$, and
$y$ be the image of the closed point
$\eta$ of $V$ by the morphism $h:V \to X \times_{S} X$.
Set $f=\pi_{1}h$ and $g=\pi_{2}h$,
where $\pi_{i}:X\times_{S} X \to X$ is the $i$-th projection.
Since $X$ is $\mathcal{I}$-separated over $S$, $f$ and $g$ must
coincide. Therefore, $h$ factors through
$X$, since $\Delta:X \to X \times_{S} X$ is the equalizer
of $\pi_{1},\pi_{2}$. This shows that $y$ is in the image of $\Delta$.
\end{proof}

\begin{Prop}
\label{prop:val:crit:proper}
Suppose $\mathcal{I}$ is a family of morphisms,
parametrizing specializations.
Let $g:X \to S$ be a morphism of $\scr{A}$-schemes.
Then, the followings are equivalent:
\begin{enumerate}[(i)]
\item $X$ is universally closed, i.e.
$X \times_{S} T \to T$ is closed for any
$\scr{A}$-scheme $T$ over $S$.
\item $X$ is $\mathcal{I}$-universally closed over $S$.
\end{enumerate}
\end{Prop}
\begin{proof}
(i)$\Rightarrow$(ii):
Suppose the following commutative square is given:
\[
\xymatrix{
U \ar[r] \ar[d] & X \ar[d]^{f} \\
V \ar[r] & S
}
\]
where $U \to V$ is a morphism in $\mathcal{I}$.
Let $U \to V \times_{S} X$ be the induced morphism,
$p_{0} \in V \times_{S} X$ be the image of $U$, and $Z=\overline{\{p_{0}\}}$
be the closed subset with the reduced induced subscheme
structure. Since $X$ is universally closed,
the image of $Z \to V$ is closed, hence
there is a section $\iota:V \to Z$ by condition (4) of \ref{def:morph:par:spe}.
Composing $\iota$ with $Z \to X$ gives the required morphism.

(ii)$\Rightarrow$(i):
Since universally-closedness is stable under
pullbacks, it suffices to show that $f:X \to S$ is closed,
i.e. the image of $f$ is stable under specializations.
Let $f(x) \rightsquigarrow s$ be a specialization
on $S$. Set $Z=\overline{\{x\}} \subset X$
and $W=\overline{\{f(x)\}} \subset S$
be closed subsets, with the reduced induced subscheme
structures.
Then, condition (3) of \ref{def:morph:par:spe}
implies that there exists a morphism $U \to V$ in $\mathcal{I}$
and a commutative diagram
\[
\xymatrix{
U \ar[r] \ar[d] & Z \ar[r] \ar[d] & X \ar[d] \\
V \ar[r] & W \ar[r] & S
}
\]
with $s$ in the image of $V \to S$.
Since $X$ is $\mathcal{I}$-universally closed,
there exists a morphism $V \to X$
making the whole diagram commutative.
Since the generic point $\xi$ of $V$ is contained in $Z$ and $V$
is reduced, this morphism factors through $Z$.
This shows that $s$ is in the image of $Z \to S$.
\end{proof}
Now, we will give a family of morphisms
which parametrizes specializations.
\begin{Def}
\label{def:morph:i_0}
Let $\mathcal{I}_{0}$ be a family of morphisms
$f:U \to V$ such that:
\begin{enumerate}
\item There is a valuation ring $R$,
and $V=\tilde{\Spec} R=\{\xi, \eta\} \subset \Spec R$ with the
induced topology,
where $\xi$ and $\eta$ are the generic point
and the closed point of $\Spec R$, respectively.
\item The structure sheaf $\scr{O}_{V}$ of $V$ is defined
as follows: the ring of global sections is $R$,
and $\scr{O}_{V,\xi}=K$, where $K$ is the
fractional field of $R$.
$\beta_{V}:\alpha_{1}\scr{O}_{V} \to \tau_{V}$
is defined by 
\[
\mathfrak{a} \mapsto \begin{cases}
1 & (\mathfrak{a}=R) \\
\{\eta \} & (0 \neq \mathfrak{a} \leq \mathfrak{M}_{\eta}) \\
0 & (0=\mathfrak{a})
\end{cases}
\]
\item $U$ is the spectrum of $K$, and $f:U \to V$
is the canonical inclusion.
\end{enumerate}
\end{Def}

It is obvious that
quasi-compact open coverings are $\mathcal{I}_{0}$-coverings.
\begin{Prop}
The above $\mathcal{I}_{0}$ parametrizes
specializations.
\end{Prop}
\begin{proof}
We will verify the condition of Definition
\ref{def:morph:par:spe}.
\begin{enumerate}
\item Obvious from the definition.
\item The maps between the underlying spaces
obviously coincide, hence we only have
to show that the two maps
$g^{\#},h^{\#}:\scr{O}_{X,x} \rightrightarrows \scr{O}_{V,\eta}$
coincide, where $x$ is the image of $\eta$.
Set $\iota:\scr{O}_{V,\eta} \to \scr{O}_{V,\xi} \simeq K$.
This is injective, hence $\iota g^{\#}=\iota h^{\#}$
shows that $g^{\#}=h^{\#}$.
\item Let $g:X \to Y$ be a dominant morphism
of two reduced irreducible $\scr{A}$-schemes,
and $x_{0},y_{0}$ be the generic points of $X,Y$,
respectively. $x_{0}$ maps to $y_{0}$
by this morphism, since it is dominant.
Let $y_{1}$ be any point of $Y$.
We have a injective morphism
\[
\scr{O}_{Y,y_{1}} \hookrightarrow \scr{O}_{Y,y_{0}}
\hookrightarrow \scr{O}_{X,x_{0}}.
\]
Set $K=\scr{O}_{X,x_{0}}$.
Then, there is a valuation ring $R$ of $K$,
dominating $\scr{O}_{Y,y_{1}}$.
This gives a morphism 
$u:U=\Spec K \to X$ and $v:V=\tilde{\Spec} R \to Y$
making the following diagram commutative:
\[
\xymatrix{
U \ar[r]^{u} \ar[d] & X \ar[d] \\
V \ar[r]_{v} & Y
}
\]
satisfying $u(\xi)=x_{0}$ and $v(\eta)=y_{1}$.
\item Let $V=\tilde{\Spec} R$, and suppose $U \to V$
factors through a reduced irreducible $\scr{A}$-scheme $Z$,
with $U \to Z$ dominant and $g:Z \to V$ surjective.
Let $z_{0}$ be the generic point of $Z$,
and $z_{1}$ be a point in $Z$ such that $g(z_{1})=\eta$.
Then, we have a commutative diagram of dominating morphisms
\[
\xymatrix{
\scr{O}_{V,\eta} \ar[d] \ar[r] & \scr{O}_{Z,z_{1}} \ar[d] \\
\scr{O}_{V,\xi} \ar[r] & \scr{O}_{Z,z_{0}} \ar[r] & \scr{O}_{V,\xi}
}
\]
All arrows are injective.
Since the composition of the second row is the identity,
we have $\scr{O}_{V,\xi} \simeq \scr{O}_{Z,z_{0}}$.
Since a valuation ring is maximal among dominating morphisms,
the arrow of the upper row must be an isomorphism.
This gives a section $V \to Z$ of $g$.
\end{enumerate}
\end{proof}

\begin{Rmk}
\label{rmk:val:crit:integral}
Suppose $X$ is irreducible and reduced,
and $K$ is the function field of $X$.
In this case, we can strengthen the
valuative criteria as follows:
let $\mathcal{I}_{1}=\{\Spec K \to V\}$
be the subfamily of $\mathcal{I}_{0}$,
consisting of all morphisms the sources of which
are $\Spec K$.
Then,
\begin{enumerate}
\item $X$ is separated over $S$
if and only if $X$ is $\mathcal{I}_{1}$-separated over $S$.
\item $X$ is universally closed over $S$
if and only if $X$ is $\mathcal{I}_{1}$-universally closed over $S$.
\end{enumerate}
This is easily seen, by taking $x$ in the proofs
of Proposition \ref{prop:val:crit:sep} and Proposition \ref{prop:val:crit:proper}
 as the generic point of $X$.
We will make use of this observation in Subsection \ref{subsec:classical:ZR}.
\end{Rmk}

\begin{Rmk}
In this article, we only use $\mathcal{I}$-separatedness
and $\mathcal{I}$-properness for
describing the valuative criteria.
However, the reader may know that
other properties of morphisms can also be formulated
by the right lifting properties with respect to other
families $\mathcal{I}$ of morphisms,
even in the classical algebraic geometry.
For example, let $\mathcal{I}$ be a family of 
morphisms $X_{0} \to X$ of affine $S$-schemes,
where $X_{0}$ is a closed subscheme of $X$
defined by a nilpotent ideal.
Then, a $S$-scheme $Y$ is \textit{formally unramified}
(resp. \textit{formally smooth}, \textit{formally \'{e}tale}
if it is $\mathcal{I}$-separated 
(resp. \textit{$\mathcal{I}$-universally closed},
\textit{$\mathcal{I}$-proper}).
$Y$ is \textit{unramified} if it is formally unramified
and of finite type over $S$.
$Y$ is \textit{smooth} (resp. \textit{\'{e}tale}) if it is formally smooth 
(resp. formally \'{e}tale) and of finite
presentation over $S$ (\cite{EGA4}, \S 17).
We will treat these subjects in the future,
and the category-theoretical arguments in this paper will be its base.
\end{Rmk}

\section{Zariski-Riemann spaces}
In this section, we will construct
a universal proper $\scr{A}$-scheme
for a given $\scr{A}$-scheme,
which is known as the Riemann-Zariski space.
The construction is somewhat difficult
than the universal separated scheme constructed
previously, since we cannot use
the PQ-decomposition to bound the cardinality
of the morphisms.

\subsection{Zariski-Riemann spaces}

\begin{Prop}
\label{prop:closure:str}
Let $f:X \to Y$ be a morphism
of $\scr{A}$-schemes.
Then there exists a decomposition
$X\to Z \to Y$ of $f$, such that:
\begin{enumerate}
\item $Z \to Y$ is a closed immersion.
\item \textit{Universality}:
if $X \to \tilde{Z} \to Y$ is another decomposition of $f$
with $\tilde{Z} \to Y$ a closed immersion,
then there is a unique morphism $Z \to  \tilde{Z}$
making the whole diagram commutative:
\[
\xymatrix{
X \ar[r] \ar[d] \ar[dr] & Y \\
Z \ar[ur] \ar@{.>}[r] & \tilde{Z} \ar[u]
}
\]
\end{enumerate}
Further, the image of $f$ is dense in $Z$,
and $\scr{O}_{Z} \to \scr{O}_{X}$ is injective.
\end{Prop}
\begin{proof}
Let $\scr{S}$ be the set of all isomorphism classes
of series of morphisms $\{X \to Z_{\lambda} \to Y\}$
of $\scr{A}$-schemes, where
$Z_{\lambda} \to Y$ is a closed immersion.
Then $\scr{S}$ is small, since closed immersions
are Q-morphisms. Let $Z$ be the limit of $\{Z_{\lambda}\}$.
Then $Z \to Y$ is also a closed immersion,
since proper morphisms and Q-morphisms are stable
under taking limits.
The universality is clear from the construction.
It remains to show that $X \to Z$ is dominant.
We may assume that $X \to Y$ is a Q-morphism.
To see this, it is enough to show
that there is an $\scr{A}$-scheme structure on the closure $\overline{X}$
of $X$ in $Y$.
We define the structure sheaf $\scr{O}_{\overline{X}}$ by
the sheafification of 
\[
W \mapsto \underrightarrow{\lim}_{V+\overline{X} \geq W}
\scr{O}_{Y}(V)/
\ker[f^{\#}:\scr{O}_{Y}(V)\to \scr{O}_{X}(f^{-1}V)].
\]
Since $\alpha_{1}\scr{O}_{Y}(V) \to \tau_{Y}(W)/\overline{X}
=\tau_{\overline{X}}(W)$ factors through
$\alpha_{1}\scr{O}_{Y}(V)/\ker f^{\#}$,
we obtain the support morphism 
$\beta_{\overline{X}}:\alpha_{1}\scr{O}_{\overline{X}}
\to \tau_{\overline{X}}$.
We also have the natural morphisms
$X \to \overline{X}$ and $\overline{X} \to Y$,
which shows that $X \to Z$ is indeed dominant.
We also see that $\scr{O}_{\overline{X}} \to \scr{O}_{X}$
is injective, hence $\scr{O}_{\overline{X}} \to \scr{O}_{Z}$
is injective. On the other hand, 
$\scr{O}_{Y} \to \scr{O}_{Z}$ is stalkwise surjective,
which shows that $Z$ is actually isomorphic to $\overline{X}$
as an $\scr{A}$-scheme.
\end{proof}

\begin{Thm}
\label{thm:ZR:functor}
Fix an $\scr{A}$-scheme $S$, and
Let $\cat{prop.$\scr{A}$-Sch/S}$
be the full subcategory of $\cat{$\scr{A}$-Sch/S}$,
consisting of $\scr{A}$-schemes proper over $S$.
Then, the underlying functor
$\cat{prop.$\scr{A}$-Sch/S}\to \cat{$\scr{A}$-Sch/S}$
has a left adjoint.
\end{Thm}
\begin{proof}
Let $X$ be an $\scr{A}$-scheme over $S$,
and $\scr{S}$ be a set of isomorphism classes
of dominant $S$-morphisms $f:X \to Y$,
with $Y$ proper over $S$, and $\scr{O}_{Y} \to f_{*}\scr{O}_{X}$
injective.
From Proposition \ref{prop:closure:str}
and Freyd's adjoint functor theorem (\cite{CWM}, p121),
it suffices to show $\scr{S}$ is small.

Let $f:X \to Y$ be a dominant $S$-morphism,
with $Y$ proper over $S$. 
\begin{itemize}
\item[Step 1:]
The points $y$ of $Y$ are parametrized by the commutative squares
\[
\xymatrix{
\Spec \kappa(x) \ar[r] \ar[d] & X \ar[d] \\
V \ar[r]_{\pi} & S
}
\]
where $\Spec \kappa(x) \to V$ is a morphism in $\mathcal{I}_{0}$
which is described in Definition \ref{def:morph:i_0}.
$y$ is given by the image of the closed point
by the unique map $V \to Y$:
this is true, since the set of points in $Y$ given by the above
diagram is stable under specialization, hence closed.
On the other hand, $f$ is a dominant map, hence
any point of $Y$ must be given by the above diagram.
\item[Step 2:]
Note that the set $\{\kappa(x)\}_{x \in X}$ is a small
set. This implies that the set of isomorphism classes of
morphisms in $\mathcal{I}_{0}$
of the form $\kappa(x) \to V$, where $x \in X$,
is also small. 
Let $\eta$ be the closed point of $V$.
Then, $\scr{O}_{V,\eta}$ is also a small set,
and the morphism $V \to S$ is determined by
the map $\scr{O}_{S,\pi(x)} \to \scr{O}_{V,\eta}$.
Summing up, we see that the set of isomorphism classes 
of the above commutative squares are small. 
Since the points of $Y$ are
parametrized by these morphisms,
the set of isomorphism classes of the underlying
spaces of $Y$'s are small.
\item[Step 3:] Since $\scr{O}_{Y} \to f_{*}\scr{O}_{X}$
is injective, the set of isomorphism classes
of $Y$'s (as $\scr{A}$-schemes) is also small,
ditto for the set of morphisms $X \to Y$.
This shows that $\scr{S}$ is a small set,
and we have finished the proof.
\end{itemize}
\end{proof}

We will denote the above left adjoint functor by $\ZR_{S}$.
\begin{Rmk}
The above functor 
and its construction is known as the Stone-\v{C}ech compactification.
\end{Rmk}

\subsection{Embedding into the Zariski-Riemann spaces}
In the sequel, fix a base $\scr{A}$-scheme $S$.

First, we confirm basic facts.
\begin{Prop}
Let $X,Y$ be a scheme over $S$.
\begin{enumerate}
\item If $X \to \ZR_{S}(X)$ is a Q-morphism,
then $X$ is separated.
\item If there is a Q-morphism $X \to Y$,
with $Y$ proper over $S$,
then $X \to \ZR_{S}(X)$ is a Q-morphism.
\item If $X \to \ZR_{S}(X)$ is a Q-morphism
and $Y \to X$ is a Q-morphism, then
$Y \to \ZR_{S}(Y)$ is a Q-morphism.
\item Let $\{X^{\lambda}\}$ be a filtered projective system
of $\scr{A}$-schemes over $S$ such that 
$X^{\lambda} \to \ZR_{S}(X^{\lambda})$ is a Q-morphism
for any $\lambda$.
If $X=\underleftarrow{\lim}_{\lambda}X^{\lambda}$,
then $X \to \ZR_{S}(X)$ is also a Q-morphism.
\end{enumerate}
\end{Prop}
\begin{proof}
(1)-(3) are straightforward.
We will only prove (4).
Since $X^{\lambda} \to \ZR_{S}(X^{\lambda})$
is a Q-morphism,
$X \to \underleftarrow{\lim}_{\lambda}\ZR_{S}(X^{\lambda})$
is also a Q-morphism, by Proposition
\ref{prop:imm:filter:limit}.
This morphism factors through $\ZR_{S}(X)$
since $\underleftarrow{\lim}_{\lambda}\ZR_{S}(X^{\lambda})$
is proper, hence $X \to \ZR_{S}(X)$
is a Q-morphism.
\end{proof}

\begin{Prop}
\label{prop:open:imm:ZR}
Let $X, Y$ be an $\scr{A}$-scheme over $S$.
Let $f:X \to Y$ be a dominant Q-morphism of $\scr{A}$-schemes over $S$,
with $Y$ proper.
Then, $f$ is an open immersion if and only if 
$\iota: X \to \ZR_{S}(X)$ is.
\end{Prop}
\begin{proof}
Let $\pi: \ZR_{S}(X) \to Y$ be the canonical morphism.
Note that $\pi$ is proper, since $Y$ is proper over $S$.
Also, since $f$ is dominant, $\pi$ must be surjective.

First, we will show (*):
$\ZR_{S}(X)\setminus \iota(X)=\pi^{-1}(Y \setminus f(X))$.
It is obvious that the left-hand side contains the right-hand side,
so we will show the converse.
Assume that there exists $u \in \ZR_{S}(X)\setminus \iota(X)$
such that $\pi(u)$ is in the image of $f$,
say $\pi(u)=f(x)$.
Since $\iota$ is dominant, there is a point $\xi \in X$
such that $\iota(\xi)$ specializes to $u$.
Also, since $f$ is a Q-morphism and
\[
f(\xi)=\pi\iota(\xi) \rightsquigarrow \pi(u)=f(x),
\]
we see that $\xi$ specializes to $x$.
Let $W$, $W^{\prime}$, $W^{\prime\prime}$
be the closure of $\{\xi\}$, $\{\iota(\xi)\}$,
$\{f(\xi)\}$, respectively,
with induced reduced subscheme structures.
We have a series of local homomorphisms
\[
\scr{O}_{W^{\prime\prime},f(x)}
\to \scr{O}_{W^{\prime},\iota(x)}
\to \scr{O}_{W,x},
\]
and these are injective since $W \to W^{\prime\prime}$
is dominant.
Also, these are surjective since $f$ is a Q-morphism,
hence isomorphisms.
Therefore, the homomorphism
\[
\scr{O}_{W^{\prime},\iota(x)}
\simeq \scr{O}_{W^{\prime\prime},\pi(u)}
\to \scr{O}_{W^{\prime},u} 
\]
implies that $\scr{O}_{W^{\prime},u}$
dominates $\scr{O}_{W^{\prime},\iota(x)}$.
Let $R$ be a valuation ring of $K=\scr{O}_{W^{\prime},\iota(\xi)}$
dominating $\scr{O}_{W^{\prime},u}$.
Consider the following commutative square:
\[
\xymatrix{
\Spec K \ar[r] \ar[d] & \ZR_{S}(X) \ar[d] \\
\tilde{\Spec} R \ar[r] & S
}
\]
Then, there are two morphisms $\tilde{\Spec} R \to \ZR_{S}(X)$
which send the closed point of $\tilde{\Spec} R$
to $\iota(x)$ and $u$, respectively.
This contradicts to the fact that $\ZR_{S}(X)$ is separated over $S$.

Now, suppose $f$ is an open immersion.
Since $f$ is a Q-morphism, $\iota$ is also a Q-morphism.
Moreover, the right-hand side of (*) is closed,
hence $\iota(X)$ is open.
This implies that $\iota$ is an open immersion.

Conversely, suppose $\iota$ is an open immersion.
Then, the left-hand side of (*) is closed,
and $\pi$ being proper and surjective implies that
\[
\pi(\ZR_{S}(X) \setminus \iota(X))=Y \setminus f(X)
\]
is also closed. Hence, $f$ is an open immersion.
\end{proof}

\begin{Cor}
\label{cor:ZR:open:imm}
If $Y \to X$ is a closed (resp. open) immersion,
and $X \to \ZR_{S}(X)$ is an open immersion, 
then $Y \to \ZR_{S}(Y)$ is an open immersion.
\end{Cor}
\begin{proof}
Let $\overline{Y}$ be the closure of $Y$
in $\ZR_{S}(X)$.
Then, $Y \to \overline{Y}$ is an open immersion
to a proper $\scr{A}$-scheme over $S$.
Then, Proposition \ref{prop:open:imm:ZR}
tells that $Y \to \ZR_{S}(Y)$ is also an open immersion.
\end{proof}

We want to know when
$X \to \ZR_{S}(X)$ is an open immersion
for a morphism $X \to S$ of $\scr{Q}$-schemes.
Note that the condition `of finite type' is crucial for
the open embedding. We will see from now on,
what happens if drop off the condition.

\begin{Prop}
Let $X \to S$ be a morphism between affine schemes.
Then, $X \to \ZR_{S}(X)$ is a Q-morphism.
\end{Prop}
\begin{proof}
Let $S=\Spec^{\scr{A}} A$.
It suffices to show when $X=\Spec A[x_{\lambda}]_{\lambda \in \Lambda}$,
the spectrum of the polynomial ring of coefficient ring $A$
with infinitely many variables.

For any finite subset $\Lambda^{\prime}$ of $\Lambda$,
set $X^{\Lambda^{\prime}}=\Spec A[x_{\lambda}]_{\Lambda^{\prime}}$.
These can be embedded into a proper
scheme $Y^{\Lambda^{\prime}}$ over $S$.
Even if $\Lambda_{1} \supset \Lambda_{2}$ is an inclusion
between finite subsets of $\Lambda$,
We need not have a morphism $Y^{\Lambda_{1}} \to Y^{\Lambda_{2}}$
extending $X^{\Lambda_{1}} \to X^{\Lambda_{2}}$:
we only obtain rational maps.
However, when given a fixed $\Lambda_{1}$,
blow up all the indeterminancy locus of
$Y^{\Lambda_{1}} \to Y^{\Lambda_{2}}$,
where $\Lambda_{2}$ runs through all the subset of $\Lambda_{1}$
and we obtain another proper scheme $\tilde{Y}^{\Lambda_{1}}$.
Replacing $Y^{\Lambda_{1}}$ by $\tilde{Y}^{\Lambda_{1}}$ 
for each $\Lambda_{1}$
gives a filtered projective system $\{\tilde{Y}^{\Lambda^{\prime}}\}$
of proper schemes over $S$,
extending the projective system $\{X^{\Lambda^{\prime}}\}$.
The morphisms $X^{\Lambda^{\prime}} \to \tilde{Y}^{\Lambda^{\prime}}$
are Q-morphisms, hence
\[
X=\underleftarrow{\lim}_{\Lambda^{\prime}} X^{\Lambda^{\prime}}
\to Y=\underleftarrow{\lim}_{\Lambda^{\prime}}
\tilde{Y}^{\Lambda^{\prime}}
\]
is also a Q-morphism, and $Y$ is proper.
\end{proof}

\begin{Exam}
Let $R=\ZZ[x_{n}]_{n \in \NNN}$ be a polynomial ring
with infinitely many variables, and set $\Aff^{\infty}=\Spec R$.
We will see that $\Aff^{\infty}$ cannot be embedded
as an open subscheme of a proper $\scr{A}$-scheme.
We have a surjection $R \to \QQ$,
hence there is a closed immersion $\Spec \QQ \to \Aff^{\infty}$.
We have a natural dominant immersion $\Spec \QQ \to \Spec \ZZ$,
which is not an open immersion.
This shows that $\Aff^{\infty} \to \ZR_{\ZZ}(\Aff^{\infty})$
cannot be an open immersion
by Corollary \ref{cor:ZR:open:imm}, although it is a Q-morphism.
This tells that, we may not be able to obtain an open embedding
if we drop the `of finite type' condition.
The decomposition which Temkin gave 
does not give the embedding (\cite{Tem}).
\end{Exam}
As a corollary, we obtain
\begin{Cor}
The infinite-dimensional projective space $\PP^{\infty}=\Proj R$
is not proper.
\end{Cor}
\begin{proof}
We have a natural open immersion $\Aff^{\infty} \to \PP^{\infty}$,
which shows that $\PP^{\infty}$ cannot be proper.
\end{proof}


\subsection{Classical Zariski-Riemann space as an $\scr{A}$-scheme}

So far, we have constructed
a universal compactification $\ZR_{S}(X)$
of a given scheme $X$. However, since we constructed
it by the adjoint functor theorem, it is difficult to
understand its structure. Also, the topology may
be very different from what we expect;
we already have the notion of Zariski-Riemann spaces
of a given field $K$ containing a base ring $A$,
but $\ZR_{\Spec A}(\Spec K)$ may not coincide with this
conventional one.

Therefore, we would like to construct a more
accessible $\scr{A}$-object; its topology should
be more `algebraic', so that it coincides
with the conventional one in simple cases.
These will give the class of $\scr{A}$-schemes `of profinite type',
which describe the pro-category
of ordinary schemes.

In the sequel,
we fix a field $K$,
and any $\scr{A}$-scheme $X$ is reduced
and has a dominant morphism $\Spec K \to X$.
This implies that $X$ is irreducible.
Moreover, we consider only dominant morphisms,
unless otherwise noticed.

\begin{Def}
Let $S$ be an $\scr{A}$-scheme
with a dominant morphism $\Spec K \to S$.
\begin{enumerate}
\item Set 
\[
\scr{M}_{0}^{S}=\scr{P}^{f}(C(S)_{\cpt} \times
(\scr{P}^{f}(K \setminus \{0\})\setminus \emptyset )).
\]
The addition on $\scr{M}^{S}_{0}$ is defined by taking the union.
The multiplication on $\scr{M}^{S}_{0}$
is defined by
\[
\{ (Z_{1i},\alpha_{1i})\}_{i} \cdot
\{ (Z_{2j},\alpha_{2j})\}_{j}
=\{ (Z_{1i}\cdot Z_{2j},\alpha_{1i}\cup \alpha_{2j})\}_{i,j}.
\]
Both two operations are associative and commutative,
and the addition is idempotent.
The distribution law holds,
and there is the additive unit $\mathbf{0}=\emptyset$.
This is also the absorbing element with respect to the multiplication.
However, there is no multiplicative unit,
hence $\scr{M}^{S}_{0}$ fails to be an idempotent semiring.
\item For any $(Z,\alpha) \in 
C(S)_{\cpt} \times (\scr{P}^{f}(K \setminus \{0\})\setminus \emptyset )$,
a set $Z[\alpha]$ is defined by
the subset of $S$, consisting of all points $s \in S$
which satisfies either
\begin{enumerate}[(i)]
\item $s \in Z$, or
\item The maximal ideal $\mathfrak{M}_{S,s}$
is not in the image of $\Spec \scr{O}_{S,s}[\alpha] \to \Spec \scr{O}_{S,s}$.
\end{enumerate}
\item Let $\mathfrak{a}=\{(Z_{i},\alpha_{i})\}_{i}$
and $\mathfrak{b}=\{(W_{j},\beta_{j})\}_{j}$ be 
two elements of $\scr{M}^{S}_{0}$.
We write $\mathfrak{a} \prec \mathfrak{b}$ if:
\begin{enumerate}
\item $\cap_{i} Z_{i}[\alpha_{i}] \supset \cap_{j}W_{j}[\beta_{j}]$, and
\item For any $i$ and any $s \in S\setminus Z_{i}[\alpha_{i}]$,
set $J_{s}=\{j \mid s \in S\setminus W_{j}[\beta_{j}]\}$.
Then for any map $\sigma:J_{s} \to \cup_{j \in J_{s}} \beta_{j}$
such that $\sigma_{j} \in \beta_{j}$,
$(\sigma^{-1}_{j})_{j}$ generates the unit ideal
in $\scr{O}_{S,s}[\alpha_{i}][\sigma^{-1}_{j}]_{j}$.
\end{enumerate}
This relation $\prec$ is reflective.
It is also true that $\prec$ is transitive,
but this seems to be difficult to prove it at this moment,
so we will not use this fact.
\item Define $\approx$ to be the equivalence relation
generated be the relation $\prec$,
namely: $\mathfrak{a} \approx \mathfrak{b}$ if and only if there is
a sequence $\mathfrak{a}=\mathfrak{a}_{0},\mathfrak{a}_{1},\cdots,
\mathfrak{a}_{n}=\mathfrak{b}$ of elements
of $\scr{M}^{S}_{0}$ such that $\mathfrak{a}_{i} \prec \mathfrak{a}_{i+1}$
and $\mathfrak{a}_{i} \succ \mathfrak{a}_{i+1}$ for each $i$.
Let $\scr{M}^{S}=\scr{M}^{S}_{0}/\approx$ be the
quotient set.
\end{enumerate}
\end{Def}
\begin{Prop}
The addition and the multiplication on $\scr{M}^{S}_{0}$
descend to $\scr{M}^{S}$, and $\scr{M}^{S}$
becomes a II-ring with these operations.
\end{Prop}
\begin{proof}
We will divide the proof in several steps.
\begin{itemize}
\item[Step 1:] We will show that the addition descends to $\scr{M}^{S}$.
To show this, it suffices to show that
if $\mathfrak{a}_{1} \prec \mathfrak{b}_{1}$
and $\mathfrak{a}_{2} \prec \mathfrak{b}_{2}$,
then $\mathfrak{a}_{1}+\mathfrak{a}_{2} \prec
\mathfrak{b}_{1}+\mathfrak{b}_{2}$.
Set $\mathfrak{a}_{1}=\{(Z_{i},\alpha_{i})\}_{i \leq m}$,
$\mathfrak{a}_{2}=\{(Z_{i},\alpha_{i})\}_{i > m}$,
$\mathfrak{b}_{1}=\{(W_{j},\beta_{j})\}_{j \leq n}$,
and $\mathfrak{b}_{2}=\{(W_{j},\beta_{j})\}_{j > n}$.
It is obvious that $\cap_{i}Z_{i}[\alpha_{i}] \supset \cap_{j}W_{j}[\beta_{j}]$.
Take arbitrary $i$ and $s \in S\setminus Z[\alpha_{i}]$.
We may assume $i \leq m$.
For any $\sigma:J_{s} \to \cup_{j \in J_{s}} \beta_{j}$
such that $\sigma_{j} \in \beta_{j}$,
$(\sigma_{j}^{-1})_{j \leq n}$ generates the unit ideal of
$\scr{O}_{S,s}[\sigma_{j}^{-1}]_{j \leq  n}$, since
$\mathfrak{a}_{1} \prec \mathfrak{b}_{1}$. Hence
$(\sigma_{j}^{-1})_{j}$ generates the unit ideal of
$\scr{O}_{S,s}[\sigma_{j}^{-1}]_{j}$.
This shows that $\mathfrak{a}_{1}+\mathfrak{a}_{2}
\prec \mathfrak{b}_{1} +\mathfrak{b}_{2}$.
\item[Step 2:] We will show that the addition descends to $\scr{M}^{S}$.
To show this, it suffices to show that
if $\mathfrak{a}_{1} \prec \mathfrak{a}_{2}$
and $\mathfrak{b}_{1} \prec \mathfrak{b}_{2}$,
then $\mathfrak{a}_{1}\cdot\mathfrak{b}_{1} \prec
\mathfrak{a}_{2}\cdot\mathfrak{b}_{2}$.
Set $\mathfrak{a}_{1}=\{(Z_{i},\alpha_{i})\}_{i \leq m}$,
$\mathfrak{a}_{2}=\{(Z_{i},\alpha_{i})\}_{i > m}$,
$\mathfrak{b}_{1}=\{(W_{j},\beta_{j})\}_{j \leq n}$,
and $\mathfrak{b}_{2}=\{(W_{j},\beta_{j})\}_{j > n}$.
Since $\cap_{i\leq m} Z_{i}[\alpha_{i}] \supset \cap_{i> m} Z_{i}[\alpha_{i}]$
and $\cap_{j \leq n} W_{j}[\beta_{j}] \supset \cap_{j >n} W_{j}[\beta_{j}]$,
we have
\[
\cap^{i \leq m}_{j \leq n}
Z_{i}\cdot W_{j}[\alpha_{i} \cup \beta_{j}] \supset
 \cap^{i > m}_{j > n}
Z_{i}\cdot W_{j}[\alpha_{i} \cup \beta_{j}]
\]
For any $i_{0} \leq m,j_{0} \leq n$ and any
$s \in Z_{i_{0}}\cdot W_{j_{0}}[\alpha_{i_{0}} \cup \beta_{j_{0}}]$,
set
\[
J_{s}=\{ (i,j) \mid i >m, j >n, s\in S\setminus Z_{i}\cdot W_{j}[\alpha_{i}
\cup \beta_{j}]\}.
\]
Let $\sigma:J_{s} \to \cup^{i>m}_{j>n}
(\alpha_{i} \cup \beta_{j})$
be a map such that $\sigma_{ij} \in \alpha_{i} \cup \beta_{j}$.
Suppose for any $i>m$, there exists $j=j(i)>n$ such that
$\sigma_{ij(i)} \in \alpha_{i}$.
Then, $(\sigma^{-1}_{ij(i)})_{i}$ generates the unit ideal
in $\scr{O}_{S,s}[\alpha_{i_{0}}]
[\sigma^{-1}_{ij(i)}]_{i}$, hence
$(\sigma^{-1}_{ij})_{ij}$ generates the unit ideal
in $\scr{O}_{S,s}[\alpha_{i_{0}} \cup \beta_{j_{0}}]
[\sigma^{-1}_{ij}]_{ij}$.
On the other hand, if there is a $i_{1}>m$ such that
$\sigma_{i_{1}j} \in \beta_{j}$ for all $j >n$,
then $(\sigma^{-1}_{i_{1}j})_{j}$
generates the unit ideal in
$\scr{O}_{S,s}[\beta_{j_{0}}]
[\sigma^{-1}_{i_{1}j}]_{j}$, which leads us to the same
conclusion as above.
This shows that $\mathfrak{a}_{1}\cdot\mathfrak{b}_{1}
\prec \mathfrak{a}_{2} \cdot\mathfrak{b}_{2}$.
\item[Step 3:] Set $\mathbf{1}=\{(1,\{1\})\}$.
It is obvious that $\mathfrak{a} \prec \mathbf{1}$
for any $\mathfrak{a} \in \scr{M}^{S}_{0}$.
This shows that $\mathbf{1}$ is the absorbing
element with respect to the addition.

We will show that $\mathbf{1}$ is the multiplicative
unit. It suffices to show that $\mathfrak{a}\prec \mathbf{1}\cdot \mathfrak{a}$.
Set $\mathfrak{a}=\{(Z_{i},\alpha_{i})\}_{i}$.
Then, $\mathbf{1}\cdot\mathfrak{a}=\{(Z_{i},\alpha_{i} \cup \{1\})\}_{i}$.
Note that $Z_{i}[\alpha_{i} \cup \{1\}]
=Z_{i}[\alpha_{i}]$.
For any $i$ and $s \in S\setminus Z_{i}[\alpha_{i}]$,
set $J_{s}=\{j \mid s \in S\setminus Z_{j}[\alpha_{j}]\}$,
and let $\sigma:J_{s} \to \cup_{j \in J_{s}} (\alpha_{j} \cup \{1\})$
be any map with $\sigma_{j} \in \alpha_{j} \cup \{1\}$.
Here, we see that $(\sigma^{-1}_{j})_{j}$
generates the unit ideal of $\scr{O}_{S,s}[\alpha_{i}][\sigma^{-1}_{j}]_{j}$
in any case.
\item[Step 4:]
It remains to prove that the multiplication on $\scr{M}^{S}$
is idempotent.
To see this, it suffices to show that $\mathfrak{a} \prec\mathfrak{a}^{2}$
for any $\mathfrak{a} \in \scr{M}^{S}_{0}$.
Set $\mathfrak{a}=\{(Z_{i},\alpha_{i})\}_{i}$. Then
\[
\mathfrak{a}^{2}=\{(Z_{i} \cdot Z_{j},\alpha_{i} \cup \alpha_{j})\}_{i,j}
\supset \{(Z_{i},\alpha_{i})\}_{i}=\mathfrak{a}.
\]
This shows that $\mathfrak{a} \prec\mathfrak{a}^{2}$.
\end{itemize}
\end{proof}
Note that, $\prec$ and $\leq$ coincide in $\scr{M}^{S}$.
From now on, we just write $(Z,\alpha)$
instead of $\{(Z,\alpha)\}$ for brevity.
\begin{Def}
\begin{enumerate}
\item There is a natural homomorphism
\[
C(S)_{\cpt} \to \scr{M}^{S} \quad (Z \mapsto (Z,\{1\}))
\]
of II-rings. This induces a morphism
$|\pi|:\Spec \scr{M}^{S} \to |S|$ of coherent spaces.
\item Let $p$ be an element of $\Spec \scr{M}^{S}$,
and $s=|\pi|(p)$. Set
\[
R_{p}=\scr{O}_{S,s}[a \in K \mid (1,\{a\}) \nleq p].
\]
\end{enumerate}
\end{Def}
\begin{Prop}
\label{prop:ZR:isom1}
\begin{enumerate}
\item $R_{p}$ is a valuation ring of $K$.
\item For any $a \in K\setminus \{0\}$,
$(1,\{a\}) \leq p$ if and only if $a \notin R_{p}$.
\item $R_{p}$ dominates $\scr{O}_{S,s}$.
\end{enumerate}
\end{Prop}
\begin{proof}
\begin{enumerate}
\item Assume that there is an element $a \in K\setminus \{0\}$
such that neither $a$ nor $a^{-1}$ is in $R_{p}$.
Then this implies $(1,\{a\}), (1,\{a^{-1}\}) \leq p$.
Hence, 
\[
\mathbf{1}=(1,\{a\})+(1,\{a^{-1}\}) \leq p
\]
which contradicts to $p$ being prime.
\item It suffices to show the `only if' part.
Assume that there is a $a \in R_{p}$ such that
$(1,\{a\}) \leq p$.
Then, there are a finite number of $x_{i}$'s such that
$(1,\{x_{i}\}) \nleq p$ and $a \in \scr{O}_{S,s}[x_{i}]_{i}$.
This is equivalent to saying that $a^{-1}$ is invertible
in $\scr{O}_{S,s}[x_{i}]_{i}[a^{-1}]$, hence
\[
\prod_{i}(1,\{x_{i}\} )=(1,\{x_{i}\}_{i} ) \leq (1,\{a\} ) \leq p.
\]
Since $p$ is prime, at least one of the $(1, \{x_{i}\})$'s
must be in $p$, but this is a contradiction.
\item It suffices to show
$\mathfrak{M}_{S,s} \subset \scr{O}_{S,s} \cap \mathfrak{M}_{p}$,
where $\mathfrak{M}_{S,s}$ and $\mathfrak{M}_{p}$
are maximal ideals of $\scr{O}_{S,s}$ and $R_{p}$,
respectively.
Assume there is an element $a \in \mathfrak{M}_{S,s} \setminus
\mathfrak{M}_{p}$.
Then, $a^{-1} \in R_{p}$, hence $(1,\{a^{-1}\}) \nleq p$.
On the other hand, $(1,\{a^{-1}\})=(\beta_{S}(a),\{1\})$
from the proof of Proposition \ref{prop:asch:loc}.
Also, $\beta_{S}(a) \leq s$, since $a \in \mathfrak{M}_{S,s}$.
Combining these, we have $(1,\{a^{-1}\}) \leq p$,
a contradiction.
\end{enumerate}
\end{proof}
\begin{Def}
\begin{enumerate}
\item
Let $\ZR^{f}(K,S)$ be a set of triples $(s,R,\phi)$,
where $s \in S$, $R$ is a valuation ring of $K$,
and $\phi:\scr{O}_{S,s} \to R$ is a dominant homomorphism.
\item The above proposition gives a map
$\varphi:\Spec \scr{M}^{S} \to \ZR^{f}(K,S)$
defined by $p \mapsto (\pi(p),R_{p},\phi)$,
where $\phi:\scr{O}_{S,\pi(p)} \to R_{p}$
is the natural homomorphism.
\item Conversely,
if we are given an element $R=(s,R,\phi)$ of $\ZR^{f}(K,S)$,
then set $p_{R} \in (\scr{M}^{S})^{\dagger}$
as the ideal generated by $\{(Z,\{1\})\}_{Z \leq s}$
and $\{(1,\{x\})\}_{x\notin R}$.
\end{enumerate}
\end{Def}
\begin{Prop}
\label{prop:ZR:isom2}
\begin{enumerate}
\item Let $Z$ be a closed subset of $S$, with a quasi-compact open complement,
and $\alpha$ be a non-empty subset of $K\setminus \{0\}$.
Then, $(Z, \alpha) \leq p_{R}$ if and only if $Z \leq s$, or $\alpha \not\subset R$.
\item The ideal $p_{R}$ is prime.
Thus, we have a map $\psi:\ZR^{f}(K,S) \to \Spec (\scr{M}^{S})^{\dagger}$.
\item $\varphi$ is bijective, and the inverse is $\psi$.
\end{enumerate}
\end{Prop}
\begin{proof}
\begin{enumerate}
\item The `if' part is obvious.
Suppose $(Z,\alpha) \leq p_{R}$ with $Z \nleq s$
and $\alpha \subset R$.
Then, $(Z,\alpha) \prec \{(Z_{i},1)\}^{i\leq m}_{Z_{i} \leq s}
\cup\{(1,\{b_{i}\})\}^{i>m}_{b_{i} \notin R}$.
Since $\scr{O}_{S,s}[\alpha] \subset R$,
we have $s \notin Z[\alpha]$.
This implies that
$J_{s}\subset \{i \mid i>m\}$.
Hence, $(b_{i}^{-1})_{i}$ generate the unit ideal of
$\scr{O}_{S,s}[\alpha][b_{i}^{-1}]_{i}$.
But since $b_{i} \notin R$ for any $i$, $b_{i}^{-1}$ must be in the
maximal ideal $\mathfrak{M}_{R}$ of $R$,
a contradiction.
\item Suppose $(Z,\alpha),(W,\beta) \notin p_{R}$.
Then (1) tells that $Z \nleq s$ and $W \nleq s$.
Since $s$ is a prime ideal of $C(S)_{\cpt}$,
we have $Z \cdot W \nleq s$.
Also, $\alpha \subset R$ and $\beta \subset R$
implies that $\alpha \cup \beta \subset R$.
Combining these, we have $(Z,\alpha)\cdot(W,\beta) \notin p_{R}$.
It is obvious that $\mathbf{1} \nleq p_{R}$,
hence $p_{R}$ is a prime ideal.
\item First, we show that $\psi \circ \varphi$ is the identity.
Let $p$ be any element of $\Spec (\scr{M}^{S})^{\dagger}$.
Then,
\[
(Z,\alpha) \in p \Leftrightarrow
\text{$Z \leq s$ or $\alpha \not\subset \varphi(p)$}\Leftrightarrow
(Z,\alpha) \in \psi\varphi(p).
\]
Next, we show that $\varphi \circ \psi$ is the identity.
Let $(s,R, \phi)$ be any element of $\ZR^{f}(K,S)$.
Then,
\[
a \in R \Leftrightarrow (1,\{a\}) \notin \psi(R)
\Leftrightarrow a \in R_{\psi(R)}.
\]
Also, It is obvious that $\pi(\psi(R))=s$,
so that $\varphi\circ \psi (R)=R$.
\end{enumerate}
\end{proof}
\begin{Rmk}
By Proposition \ref{prop:ZR:isom1} and Proposition \ref{prop:ZR:isom2},
we can give a topology on $\ZR^{f}(K,S)$
induced from $\Spec \scr{M}^{S}$.
We can see that the topology has an open basis
of the form $U(Z,\alpha)$, where
$Z \in C(S)_{\cpt}$, $\alpha \in \scr{P}^{f}(K\setminus \{0\})\setminus
\emptyset$, and
\[
U(Z,\alpha)=\{(s,R,\phi)\in \ZR^{f}(K,S) \mid
Z \nleq s, \alpha \subset R\}.
\]
By this topology, $\ZR^{f}(K,S)$ becomes a coherent space.
From this, we can also see that $\prec$ in $\scr{M}^{S}_{0}$
is in fact transitive.
Note also that this definition of Zariski-Riemann space
coincides with the usual definition
when $S$ is an affine $\scr{Q}$-scheme.
\end{Rmk}
\begin{Def}
Set $X=\ZR^{f}(K,S)$.
\begin{enumerate}
\item The structure sheaf $\scr{O}_{X}$ on $X$
is defined by
\[
U \mapsto \{ a \in K \mid \text{$a \in R_{p}$ for any $p \in U$.}\}
\] 
It is obvious that this is in fact a sheaf.
\item The support morphism
$\beta_{X}:\alpha_{1}\scr{O}_{X} \to \tau_{X}$
is defined by
\[
\Gamma(U,\alpha_{1}\scr{O}_{X}) \ni (f_{i})_{i}
\mapsto \{(1,\{f_{i}^{-1}\})\}_{i},
\]
where $f_{i}$'s are non-zero generators. 
\end{enumerate}
\end{Def}
We will verify that $(X,\scr{O}_{X},\beta_{X})$
is an $\scr{A}$-scheme.
\begin{Prop}
\begin{enumerate}
\item The support morphism $\beta_{X}$ is well defined.
\item The restriction maps reflect localizations.
\item For any $p \in X$, $\scr{O}_{X,p}=R_{p}$.
\end{enumerate}
\end{Prop}
\begin{proof}
\begin{enumerate}
\item It suffices to show that
$\{(1,\{f_{i}^{-1}\})\}_{i} \prec \{(1,\{g_{j}^{-1}\})\}_{j}$
if $(f_{i})_{i} \leq (g_{j})_{j}$.
Assume $\{(1,\{f_{i}^{-1}\})\}_{i} \not\prec \{(1,\{g_{j}^{-1}\})\}_{j}$.
This implies that there is a valuation ring $R$ of $K$
with $f_{i}^{-1} \in R$ for some $i$,
and $g_{j}^{-1} \notin R$ for any $j$.
This is equivalent to $g_{j} \in \mathfrak{M}_{R}$.
$(f_{i})_{i} \leq (g_{j})_{j}$ tells that
$f_{i}^{m}=\sum_{j}a_{ij}g_{j}$ for
some $m$ and some $a_{ij} \in \Gamma(U,\scr{O}_{X})$.
Let $s \in S$ be a point corresponding to $R$.
Then $(g_{j})_{j}$ generates the unit ideal
in $\scr{O}_{S,s}[f_{i}^{-1}][g_{j}]_{j}$,
but this cannot happen since $g_{j} \in \mathfrak{M}_{R}$.
\item Let $V \subset U$ be an inclusion
of quasi-open subsets of $X$,
and $Z=U \setminus V$ be the closed subset of $U$.
Let $f \in \Gamma(U,\scr{O}_{X})$
be a section with $\beta_{X}(f) \geq Z$.
This implies that $f^{-1} \in R_{p}$ for any $p \in V$,
hence $f$ is invertible in $\Gamma(V,\scr{O}_{X})$.
\item It is obvious that $\scr{O}_{X,p} \subset R_{p}$.
For the converse, let $a \in R_{p}$ be any element.
Then, the closed set $Z$ corresponding to $\{(1, \{a\})\}$
does not contain $p$.
Let $U$ be the complement of $Z$.
Then $a \in \Gamma(U,\scr{O}_{X}) \subset \scr{O}_{X,p}$.
\end{enumerate}
\end{proof}
\begin{Rmk}
\begin{enumerate}
\item
There is an alternative way of defining the structure sheaf
$\scr{O}_{X}$: namely, $\scr{O}_{X}:\scr{M}^{S} \to \cat{Rng}$
is defined by
\[
\{(Z_{i},\alpha_{i})\}_{i}
\mapsto \cap_{i}\cap_{s\in S\setminus Z_{i}}
\ICL(K;\scr{O}_{S,s}[\alpha_{i}]),
\]
where $\ICL(K;\scr{O}_{S,s}[\alpha_{i}])$
is the integral closure of $\scr{O}_{S,s}[\alpha_{i}]$ in $K$.
We can easily see that this definition is equivalent
to the previous one, once we know that
the integral closure of a given domain is the intersection
of all valuation rings containing it.
This implies that we can characterize
the Zariski-Riemann space \textit{without using the notion
of valuation rings}.
However, the arguments get longer when we try to
prove other properties, if we start from this definition.
\item
Note that $\ZR^{f}(\Spec K,\Spec A)$ coincides
with the conventional Zariski-Riemann space,
if $A$ is a subring of $K$;
there is a 1-1 correspondence between
points of $\ZR^{f}(\Spec K,\Spec A)$
and valuations rings of $K$ containing $A$.
Its open basis is given by the form
$U(a_{1},\cdots,a_{n})$,
where $a_{1},\cdots a_{n}$ are elements of $K$
and $U(a_{1},\cdots,a_{n})$ is the set of
valuation rings containing $A[a_{1},\cdots,a_{n}]$.
See \cite{Matsumura} for example.
\end{enumerate}
\end{Rmk}

\subsection{Zariski-Riemann space as a functor}
\label{subsec:classical:ZR}
Now, we focus on the map
$\pi:\ZR^{f}(K,S) \to S$.
We have already seen that $|\pi|:|\ZR^{f}(K,S)| \to |S|$ is well defined
as a morphism of coherent spaces.
We will see here that $\pi$ is well defined
as a morphism of $\scr{A}$-schemes.
\begin{Prop}
\begin{enumerate}
\item
The canonical inclusion
\[
\pi^{\#}:\Gamma(U,\scr{O}_{S}) \ni a \mapsto
a \in \Gamma(\pi^{-1}U,\scr{O}_{X})
\]
gives a morphism
$\pi:\ZR^{f}(K,S) \to S$ of $\scr{A}$-schemes.
\item $\pi$ is a P-morphism:
in particular, $\pi$ is surjective.
\item $\pi$ is proper.
\end{enumerate}
\end{Prop}
\begin{proof}
\begin{enumerate}
\item
In order to see that $(\pi,\pi^{\#})$ is a morphism
of $\scr{A}$-schemes, it suffices to see that the diagram
\[
\xymatrix{
\alpha_{1}\scr{O}_{S,s} \ar[r]^{\pi^{\#}} \ar[d]_{\beta_{S}} &
\pi_{*}\alpha_{1}\scr{O}_{X} \ar[d]^{\beta_{X}} \\
\tau_{S} \ar[r] &
\pi_{*}\tau_{X}
}
\]
is commutative: namely,
we must see that
$\{(1,\{f_{i}^{-1}\})\}_{i}=\{(\beta_{S}(f_{i})_{i},1)\}$
for $(f_{i})_{i} \in \alpha_{1}\scr{O}_{S}$.
Let $p$ be any element in $\cup_{i}U(1,\{f_{i}^{-1}\})$.
Then there is a dominant morphism
$\scr{O}_{S,\pi(p)} \to R_{p}$ with $f_{i}^{-1} \in R_{p}$
for some $i$. Assume that $s=\pi(p) \in \beta_{S}(f_{i})_{i}$.
This is equivalent to saying that $f_{i}$' are
in the maximal ideal of $\scr{O}_{S,s}$.
But this contradicts to $\scr{O}_{S,\pi(s)} \to R_{p}$
being dominant.
The converse can be proven similarly.
\item It suffices to show that
$C(S)_{\cpt} \to \Spec \scr{M}^{S}$
is injective, but this is obvious from the definition.
\item From Remark \ref{rmk:val:crit:integral},
it suffices to show that when we are given a commutative
square
\[
\xymatrix{
\Spec K \ar[r] \ar[d] & X\ar[d]^{\pi} \\
\tilde{\Spec} R \ar[r]_{h} & S
}
\]
we have a unique morphism
$\tilde{\Spec} R \to X$ making the whole diagram commutative.
Set $s=h(\eta)$, where $\eta$ is the closed point of
$\tilde{\Spec} R$.
Then, the dominant morphism
$\scr{O}_{S,s} \to R$ determines a point $x$ of $X=\ZR^{f}(K,S)$.
The isomorphism $\scr{O}_{X,x} \to R$ gives
the required morphism $\tilde{\Spec} R \to X$.
\end{enumerate}
\end{proof}
\begin{Rmk}
The reader may notice that
(2) follows immediately from (3) and $\pi$ being dominant.
However, we gave a different proof here,
since the valuative criterion already uses the fact of (2).
\end{Rmk}
From now on, we refer to this proper morphism
$\pi_{S}:\ZR^{f}(K,S) \to S$ as the
\textit{classical Zariski-Riemann space associated to $S$}.
\begin{Def}
Let $T$ and $S$ be $\scr{A}$-schemes,
and $\pi_{Y}:Y=\ZR^{f}(K,T) \to T$,
$\pi_{X}:X=\ZR^{f}(K,S) \to S$ be the
associated classical Zariski-Riemann spaces.
A morphism $f:T\to S$ of $\scr{A}$-schemes induces a morphism
$\tilde{f}:Y \to X$
of $\scr{A}$-schemes as follows:
\begin{enumerate}
\item The morphism $|\tilde{f}|:|Y| \to |X|$ of the
underlying spaces are defined by
\[
(t,R,\phi) \mapsto (f(t),R,\phi\circ f^{\#}),
\]
where $f^{\#}:\scr{O}_{S,s} \to \scr{O}_{T,t}$
is the dominant morphism.
In terms of II-rings, this can be expressed as
\[
\scr{M}^{S} \ni \{(Z,\alpha)\} \mapsto 
\{(f^{-1}Z,\alpha)\} \in \scr{M}^{T},
\]
which shows that $|\tilde{f}|$ is indeed a quasi-compact morphism.
\item The morphism 
$\tilde{f}^{\#}:\scr{O}_{X} \to |\tilde{f}|_{*}\scr{O}_{Y}$
is defined by the canonical inclusion.
It is easy to see that $|\tilde{f}|_{*}\beta_{Y}\circ \tilde{f}^{\#}
= \tilde{f}^{-1}\circ \beta_{X}$.
\end{enumerate}
Hence, the map $S \mapsto \ZR^{f}(K,S)$ induces a functor
\[
\ZR^{f}(K,\cdot ):\cat{$\scr{A}$-Sch} \to
\cat{proper morphism of $\scr{A}$-schemes}.
\]
We will call this functor the \textit{ZR functor}.
\end{Def}

\subsection{Morphisms of profinite type}
Our next aim is to express
the ZR functor as a left adjoint,
namely to clarify the universal property
of the classical Zariski-Riemann space.

\begin{Prop}
Let $f:T \to S$ be a morphism of $\scr{A}$-schemes,
and $\tilde{f}:\ZR^{f}(K,T) \to \ZR^{f}(K,S)$
be the induced morphism.
Then, $f$ is separated (resp. universally closed, proper)
if and only if $\tilde{f}$ is injective (resp. surjective, bijective)
on the underlying spaces.
\end{Prop}
This is just the translation of the valuative criteria,
so we will omit  the proof.

\begin{Def}
Let $f:T \to S$ be a morphism of $\scr{A}$-schemes,
and $\tilde{f}:\ZR^{f}(K,T) \to \ZR^{f}(K,S)$
be the induced morphism.
\begin{enumerate}
\item $f$ is \textit{of profinite type}, if $\tilde{f}$ is a Q-morphism.
\item $f$ is \textit{strongly of profinite type},
if $\tilde{f}$ is an open immersion.
\end{enumerate}
\end{Def}
Of course, $f$ is separated if $f$ is of profinite type.
The next characterization of morphism of profinite type
is obvious. 
\begin{Prop}
\label{cond:profinite:eq}
Let $f:T \to S$ be a morphism of $\scr{A}$-schemes.
The followings are equivalent:
\begin{enumerate}[(i)]
\item $f$ is of profinite type.
\item For every $Z \in C(T)_{\cpt}$,
there exists $\{(Z_{i},\alpha_{i})\}_{i} \in 
\scr{P}^{f}(C(S)_{\cpt} \times (\scr{P}^{f}(K \setminus \{0\}) \setminus
\emptyset))$ such that:
\begin{enumerate}
\item $\cap_{i}(f^{-1}Z_{i})[\alpha_{i}]=Z$.
\item For every $t \in Z$,
set $I_{t}=\{i \mid t \in f^{-1}Z_{i}\}$.
Then for any map $\sigma:I_{t} \to \cup_{i \in I_{t}} \alpha_{i}$
such that $\sigma_{i} \in \alpha_{i}$,
$(\sigma^{-1}_{i})_{i}$ generates the unit ideal of
$\scr{O}_{T,t}[\sigma^{-1}_{i}]_{i}$.
\end{enumerate}
\end{enumerate}
\end{Prop}
Roughly speaking, an $\scr{A}$-scheme $X$ of profinite type
over $S$ has the coarsest topology,
which makes the map $X \to S$ quasi-compact,
 and the domain of meromorphic functions
are quasi-compact open: here,
a \textit{domain} of a meromorphic function $a \in K$
is $\{x \in X \mid a \in \scr{O}_{X,x}\}$. 

\begin{Cor}
For any $\scr{A}$-scheme $S$,
Set $X=\ZR^{f}(K,S)$.
Then, the natural morphism
$\pi_{X}:\ZR^{f}(K,X) \to X$ is an isomorphism.
\end{Cor}
\begin{proof}
It follows from Proposition \ref{cond:profinite:eq}
that $\pi_{S}:X \to S$ is of profinite type,
which is equivalent to $\pi_{X}$ being a Q-morphism.
On the other hand, $\pi_{X}$ is bijective
since $\pi_{S}$ is proper.
It is obvious that $\pi_{X}$ induces isomorphism
on each stalks.
This implies that $\pi_{X}$ is an isomorphism.
\end{proof}

We will verify some basic facts of morphisms of profinite type.
\begin{Prop}
\label{prop:basic:prop:prof}
\begin{enumerate}
\item Let $A$ be a ring, and $B$ be a finitely generated $A$-algebra.
Then, $\Spec B \to \Spec A$ is strongly profinite.
\item An open immersion is strongly profinite.
\item If $X=\cup_{i} X_{i}$ is an quasi-compact open cover
of $X$, then $\{\ZR^{f}(K,X_{i})\to \ZR^{f}(K,X)\}_{i}$
is an quasi-compact open cover of $\ZR^{f}(K,X)$.
\item Morphisms of profinite type (resp. strongly of profinite type)
are stable under compositions.
\item  Morphisms of profinite type (resp. strongly of profinite type)
is local on the base:
let $f:T \to S$ be a morphism of $\scr{A}$-schemes,
$S=\cup_{i}S_{i}$ be an quasi-compact open covering
of $S$, and $T_{i}=S_{i} \times_{S} T$.
Then, $f$ is of profinite type (resp. strongly of profinite type)
if and only if $T_{i} \to S_{i}$ is of profinite type (resp. strongly of profinite type)
for any $i$.
\item Let $f:T \to S$ be a separated morphism
of $\scr{A}$-schemes,
and $T=\cup_{i}T_{i}$ be a quasi-compact
open covering. Then $f$ is of profinite type (resp. strongly of profinite type)
if and only if $f|_{T_{i}}$ is.
\item Let $S$ be an $\scr{A}$-scheme, and
$\{X^{\lambda}\}$ be a filtered projective system
of $\scr{A}$-schemes over $S$.
Set $X=\underleftarrow{\lim}_{\lambda}X^{\lambda}$.
Then $X \to S$ is of profinite type if $X^{\lambda} \to S$ is
of profinite type for any $\lambda$.
\end{enumerate}
\end{Prop}
\begin{proof}
We will only show (1), (6) and (7);
the others are easy.
\begin{itemize}
\item[(1)] Set $B=A[x_{1},\cdots,x_{n}]$.
Then, $\ZR^{f}(K,\Spec B)$ is isomorphic to the
open set $U(1,\{x_{1},\cdots,x_{n}\})$ of $\ZR^{f}(K,\Spec A)$.
\item[(6)] The `only if' part follows from (2) and (4).
Suppose $f|_{T_{i}}$ is of profinite type.
Then $\ZR^{f}(K,T_{i}) \to \ZR^{f}(K,S)$
is Q-morphism for any $i$.
Since $f$ is separated, we see that $\ZR^{f}(K,T) \to \ZR^{f}(K,S)$
is also a Q-morphism, since $\{ZR^{f}(K,T_{i})\}_{i}$
is an quasi-compact open cover
of $\ZR^{f}(K,T)$, from (3).
\item[(7)] Since $\ZR^{f}(K,X^{\lambda}) \to \ZR^{f}(K,S)$
is a Q-morphism for any $\lambda$, and
Q-morphism is stable under taking filtered projective limits
by Proposition \ref{prop:imm:filter:limit},
it suffices to show that 
$\ZR^{f}(K,X) \to \underleftarrow{\lim}_{\lambda}\ZR^{f}(K,X^{\lambda})$
is an isomorphism.

Since $C(X)_{\cpt}=\underrightarrow{\lim} C(X^{\lambda})_{\cpt}$,
we have a natural isomorphism
\[
\underrightarrow{\lim}_{\lambda}
\scr{P}^{f}(C(X^{\lambda})_{\cpt} \times (\scr{P}^{f}(K \setminus 0)
\setminus \emptyset ))
\simeq
\scr{P}^{f}(C(X)_{\cpt} \times (\scr{P}^{f}(K \setminus 0)
\setminus \emptyset ))
\]
which yields $\underrightarrow{\lim}_{\lambda} \scr{M}^{X^{\lambda}}
\simeq \scr{M}^{X}$, namely,
$\ZR^{f}(K,X) \to \underleftarrow{\lim}_{\lambda}\ZR^{f}(K,X^{\lambda})$
is an isomorphism on the underlying spaces.
Since every stalk of both sides is a valuation ring of $K$,
the morphism of structure sheaves is also an isomorphism.
\end{itemize}
\end{proof}

\begin{Cor}
\begin{enumerate}
\item A separated morphism of $\scr{Q}$-schemes
is of profinite type.
\item A separated, of finite type morphism of $\scr{Q}$-schemes
is strongly of profinite type.
\end{enumerate}
\end{Cor}
\begin{proof}
We will just prove (1), since the proof of (2)
is similar.
Let $X \to S$ be a separated morphism of $\scr{Q}$-schemes.
From (5) and (6) of Proposition \ref{prop:basic:prop:prof},
we may assume $X=\Spec B$ and $S=\Spec A$
for some domains $A$ and $B$.
$B$ is the colimit of all finitely generated sub $A$-algebras $\{B_{\lambda}\}$,
and $\Spec B_{\lambda} \to \Spec A$ is of profinite type,
by (1). Then, (7) of Proposition \ref{prop:basic:prop:prof}
shows that $\Spec B \to \Spec A$
is also of profinite type.
\end{proof}

\begin{Def}
\begin{enumerate}
\item Let $\cat{$K$/Int}$ be the subcategory
of $\cat{$\scr{A}$-Sch}$, consisting of
irreducible reduced $\scr{A}$-schemes $X$
with a dominant morphism $\Spec K \to X$.
The morphisms in $\cat{$K$/Int}$ are dominant morphisms,
under $\Spec K$.
\item
Let $\cat{PrPf}$ be the category of proper, of profinite type morphisms
of $\cat{$K$/Int}$,
and the arrows being commutative squares.
\item There is a target functor $U_{t}:\cat{PrPf} \to \cat{$K$/Int}$,
sending $(f:X \to S)$ to $S$.
\end{enumerate}
\end{Def}

\begin{Thm}
The ZR functor $\ZR^{f}(K,\cdot)$ is the left adjoint
of $U_{t}$.
\end{Thm}
\begin{proof}
The unit $\epsilon:\Id \Rightarrow U_{t}\circ \ZR^{f}(K,\cdot)$
of the adjoint is the identity.
The counit $\eta_{X}:\ZR^{f}(K,S) \to X$
for a proper, of profinite type morphism $f:X \to S$ is given as follows:
Since $f:X \to S$ is of profinite type and proper,
$\tilde{f}:\ZR^{f}(K,X) \to \ZR^{f}(K,S)$
is an isomorphism.
Then, $\eta_{X}$ is defined by
\[
\xymatrix{
\ZR^{f}(K,S) \ar[r]^{\tilde{f}^{-1}}&
\ZR^{f}(K,X) \ar[r]^{\quad\pi_{X}} & X.
}\]
These two natural transforms $\epsilon$ and $\eta$
give the adjoint $\ZR^{f}(K,\cdot) \dashv U_{t}$.
\end{proof}

\subsection{The embedding problem revisited}

In this subsection, we will construct
a compactification functor from the ZR functor,
and characterize it by the universal property.
\begin{Def}
A Q-morphism $f:X \to Y$ of $\scr{A}$-schemes
is \textit{strict},
if for any quasi-compact open subset $U$ of $X$
and a section $a \in \scr{O}_{X}(U)$,
there exists a quasi compact open subset $V$ of $Y$
and a section $b \in \scr{O}_{Y}(V)$ such that:
\begin{enumerate}[(i)]
\item $U=f^{-1}V$, and
\item $f^{\#}(b)=a$.
\end{enumerate}
\end{Def}
In particular, an open immersion is strict.
\begin{Lem}
\label{lem:push:open}
Consider the pushout diagram
of $\scr{A}$-schemes:
\[
\xymatrix{
X \ar[r]^{f} \ar[d]_{g} & Y \ar[d]^{\tilde{g}} \\
S \ar[r]_{\tilde{f}} & T
}
\]
namely, $T=S \amalg_{X} Y$.
Suppose $f$ is an open immersion (resp. strict Q-morphism).
Then, $\tilde{f}$ is an open immersion (resp. strict Q-morphism).
\end{Lem}
\begin{proof}
We will only prove for the case $f$ is an open immersion.
In this case, $C(X)_{\cpt}$ is a localization
$(C(Y)_{\cpt})_{W}$ of $C(Y)_{\cpt}$
along some $W \in C(Y)_{\cpt}$.
Since
\[
C(T)_{\cpt}=C(S)_{\cpt} \times_{C(X)_{\cpt}} C(Y)_{\cpt},
\]
in the category of II-rings,
$C(S)_{\cpt}$ is the localization of $C(T)_{\cpt}$
along $(1,W) \in C(T)_{\cpt}$.
Hence, the map $|\tilde{f}|:|S| \to |T|$ is an open immersion
on the underlying space. 

Let us show that $\tilde{f}$ is strict Q-morphism.
Suppose $a \in \scr{O}_{S}(U)$
is a section of $S$ for a quasi-compact open set
$U$ of $S$.
Pulling back $a$ by $g$ gives a section 
$g^{\#}(a) \in \scr{O}_{X}(g^{-1}U)$.
Since $f$ is strict, 
there is a quasi-compact open $V \in Y$
and a section $b \in \scr{O}_{Y}(V)$ such that 
$f^{-1}V=g^{-1}U$ and $g^{\#}(a)=f^{\#}(b)$.
Then, $(U,V)$ gives a quasi-compact open subset of $T$,
and $(a,b) \in \scr{O}_{T}(U,V)$ gives a section.
This section $(a,b)$ maps to $a$ via $\tilde{f}^{\#}$,
hence $\tilde{f}$ is strict.
\end{proof}

\begin{Def}
\begin{enumerate}
\item
Let $T \to S$ be a dominant morphism
of irreducible, reduced $\scr{A}$-schemes,
and $K$ be the function field of $T$.
The (classical) \textit{Zariski-Riemann space} $\ZR^{f}(T,S)$ of $T \to S$
is defined by the pushout of the following:
\[
\xymatrix{
\ZR^{f}(K,T) \ar[r] \ar[d] & \ZR^{f}(K,S) \ar@{.>}[d] \\
T \ar@{.>}[r] & \ZR^{f}(T,S)
}
\]
\item Let $S$ be an irreducible reduced
$\scr{A}$-scheme.
We denote by $\cat{Int/$S$}$ the category of
irreducible, reduced $\scr{A}$-schemes
dominant over $S$, and dominant $S$-morphisms.
\item Let $f:T \to T^{\prime}$
be a morphism in $\cat{Int/$S$}$.
Then $f$ naturally induces a morphism
$\ZR^{f}(T,S) \to \ZR^{f}(T^{\prime},S)$
from the universal property of pushouts.
\end{enumerate}
\end{Def}
\begin{Prop}
\begin{enumerate}
\item
$\ZR^{f}(T,S)$ is proper and of profinite type over $S$.
\item $T \to \ZR^{f}(T,S)$ is a Q-morphism
(resp. open immersion) if $T \to S$ is of profinite type
(resp. strongly of profinite type).
\end{enumerate}
\end{Prop}
\begin{proof}
\begin{enumerate}
\item
Let $K$ be the function field of $T$,
and set $X=\ZR^{f}(T,S)$.
Applying the ZR functor $\ZR^{f}(K,\cdot)$
to the pushout diagram of the definition
of $\ZR^{f}(T,S)$ yields the following pushout diagram:
\[
\xymatrix{
\ZR^{f}(K,\ZR^{f}(K,T)) \ar[d]_{\simeq} \ar[r] &
\ZR^{f}(K,\ZR^{f}(K,S)) \ar[d] \ar[r]^{\quad\simeq} &\ZR^{f}(K,S) \\
\ZR^{f}(K,T) \ar[r] & \ZR^{f}(K,X)
}
\]
This is indeed a pushout,
since $\ZR^{f}(K,\cdot)$ is a left adjoint
and hence preserves colimits.
This shows that the right vertical arrow is also
an isomorphism, which tells that
$X \to S$ is proper and of profinite type.
\item This follows from Lemma \ref{lem:push:open}.
Note that $\ZR^{f}(K,T) \to \ZR^{f}(K,S)$
is a strict Q-morphism if $T\to S$
is of profinite type, or strongly of profinite type.
\end{enumerate}
\end{proof}
This proposition shows that
$\ZR^{f}(\cdot, S)$ is functor
from $\cat{Int/$S$}$ to the full subcategory
$\cat{PrPf/$S$}$ of $\cat{Int/$S$}$ consisting of
irreducible reduced $\scr{A}$-schemes,
proper and of profinite type over $S$.

\begin{Thm}
$\ZR^{f}(\cdot,S)$ is the left adjoint of 
the underlying functor
$U: \cat{Int/$S$} \to \cat{PrPf/$S$}$.
\end{Thm}
\begin{proof}
The unit $\epsilon_{T}:T \to \ZR^{f}(T,S)$
is the canonical morphism, for any
$T \in \cat{Int/$S$}$.
The counit $\eta_{X}:\ZR^{f}(X,S) \to X$
for a proper, of profinite type morphism $X \to S$
is defined as follows.
Consider the pushout diagram:
\[
\xymatrix{
\ZR^{f}(K,X) \ar[r] \ar[d] & \ZR^{f}(K,S) \ar[d] \\
X \ar[r]_{\iota_{X}} & \ZR^{f}(X,S)
}
\]
Since $X$ is proper and of profinite type over $S$,
the upper horizontal arrow is an isomorphism,
hence the lower arrow is also.
Define $\eta_{X}$ as the inverse of $\iota_{X}$.
These two natural transforms $\epsilon$ and $\eta$
give the adjoint $\ZR^{f}(\cdot, S )\dashv U$.
\end{proof}

In particular, we have:
\begin{Cor}
\label{cor:vari:nagata}
Let $X \to S$ be a separated morphism
of integral $\scr{Q}$-schemes.
Then, there exists a proper, of profinite type morphism
$\overline{X} \to S$ of $\scr{A}$-schemes
with a Q-morphism $\iota:X \to \overline{X}$.
Moreover, this embedding $\iota$ of $X$ is universal,
and $\iota$ is an open immersion if
$X$ is of finite type over $S$.
\end{Cor}
This is a variant of Nagata embedding (\cite{Con}).
Note that from Proposition \ref{prop:open:imm:ZR},
$X \to \ZR_{S}(X)$ becomes also an open immersion,
if $X$ is separated, of finite type over $S$.

\begin{Rmk}
There are previous constructions
of Zariski-Riemann spaces for ordinary schemes;
for example, see \cite{Tem}.
We can see without difficulty that our construction coincides
with them.
However, the proof will be somewhat technical
and takes some time to check
that these are equivalent,
since the definition is very different:
the previous one is defined by the limit space
of admissible blow ups.
The comparison will be treated in the forthcoming paper.
\end{Rmk}

\section{Appendix: The definition of $\scr{A}$-schemes}
\setcounter{subsection}{1}
In this section, we will briefly review the definition of
$\scr{A}$-schemes, when the algebraic system
$\sigma$ is that of rings. For the general definition
and detailed proofs,
we refer to \cite{Takagi}.

Before we start the definitions,
we will explain the intuitive idea and
the essential differences between $\scr{A}$-schemes and
ordinary schemes.
\begin{enumerate}
\item
The fundamental property of ordinary schemes
is that the global section functor admits the left
adjoint, namely the spectrum functor:
\[
\Spec: \cat{Ring} \rightleftarrows \cat{Sch}^{\op}: \Gamma.
\]
This enables various construction of schemes,
such as fiber products.
However, the construction of the
co-unit $X \to \Spec \Gamma(X)$ of the above adjoint
does not actually use the axiom of schemes
that it is locally isomorphic to the spectrum of a ring;
it just uses the property that the restriction
functor corresponds to localizations:
let us describe it more explicitly.
Let $\scr{O}_{X}$ be a sheaf of functions
on a space $X$. When there is a function $f \in \scr{O}_{X}$,
it determines the zero locus $\beta(f)=\{f=0\}$ on $X$.
This correspondence is the intuitive idea of the
support morphism defined below.

When the function $f$ is restricted to an open set $V$
such that $V \cap \beta(f)=\emptyset$,
then $f|_{V}$ is nowhere vanishing.
Therefore, $f$ must be invertible in $\scr{O}_{X}(V)$.
This is formulated below as the property
which we refer to as `restrictions reflect localizations'.

These setups enable us to construct the counit morphism.
This is why we put emphasis on these properties.
\item On the other hand, we stick on to
coherent underlying spaces, when defining
$\scr{A}$-schemes. This is because
coherent spaces have good properties in nature, and
we can take limits and colimits in the category
of coherent spaces.
This shows that we do not have any reason to `forget'
the coherence properties, even when we consider 
limit and colimit spaces.
We believe that this restriction is not wrong,
since we have already seen in \S1 that there are various
benefits because of this.
\end{enumerate}

\begin{Def}
\begin{enumerate}
\item An \textit{idealic semiring} is a set $R$
endowed with two operators $+$ and $\cdot$, satisfying:
\begin{enumerate}
\item $R$ is a commutative 
monoid with respect to $+$ and $\cdot$,
with two unit elements $0$ and $1$, respectively.
Further, $R$ is idempotent with respect to $+$:
$a+a=a$ for any $a \in R$.
\item The distribution law holds: $(a+b)c=ac+bc$
for any elements $a,b,c \in R$.
\item $0$ is the absorbing element with respect to the multiplication:
$0\cdot a=0$ for any $a \in R$.
$1$ is the absorbing element with respect to the addition.
\end{enumerate}
Note that an idealic semiring has a natural ordering,
defined by $a \leq b \Leftrightarrow a+b=b$.
\item An \textit{II-ring} is an idealic semiring
with idempotent multiplications.
This is conventionally called a \textit{distributive lattice},
used in Stone duality.
\item The category of II-rings are denoted by $\cat{IIRng}$.
\end{enumerate}
\end{Def}

\begin{Def}
\begin{enumerate}
\item A topological space $X$ is \textit{sober},
if any irreducible closed subset $Z$ of $X$
has a unique generic point $\xi_{Z}$,
namely, $Z=\overline{\{\xi_{Z}\}}$.
\item A sober space is \textit{coherent},
if it is quasi-compact, quasi-separated
(namely, the intersection of any two quasi-compact
open subset is again quasi-compact),
and has a quasi-compact open basis.
We denote by $\cat{Coh}$ the category
of coherent spaces and quasi-compact morphisms.
\item For a sober space $X$,
$C(X)$ is the set of all closed subsets
$Z$ of $X$.
This becomes an idealic semiring,
defining the addition as taking intersections,
and the multiplication as taking unions.
Moreover, this semiring is \textit{complete},
i.e. admits infinite summations.
The category of complete II-rings is denoted by $\cat{IIRng$^{\dagger}$}$.
\item For a coherent space $X$,
$C(X)_{\cpt}$ is the set of all closed subsets $Z$ of $X$
such that $X \setminus Z$ is quasi-compact.
This becomes an idealic semiring.
\end{enumerate}
\end{Def}
The correspondence $X \mapsto C(X)_{\cpt}$
gives an equivalence of categories 
$\cat{Coh}^{\op} \to \cat{IIRng}$:
the inverse is given by
$R \mapsto \Spec R$,
where $\Spec R$ is the set of prime ideals of $R$
with the well known topology.
This is the Stone duality.
\begin{Def}
\begin{enumerate}
\item For a ring $R$,
let $\alpha_{1}(R)$ be the set of
finitely generated ideals of $R$,
divided by the equivalence relation generated by
$I \cdot I=I$. This gives a functor
$\cat{Rng} \to \cat{IIRng}$,
where $\cat{Rng}$ is the category of rings.
\item For any ring $R$,
$\alpha_{2}:R \to \alpha_{1}(R)$ is a multiplication-preserving
map, sending $f \in R$ to the principal ideal generated by $f$.
This map gives a natural transformation,
and preserves localizations.
\item Let $X$ be a coherent space.
A $\cat{Coh}^{\op}$-valued 
(in other words, $\cat{IIRng}$-valued)
sheaf $\tau_{X}$ is defined by
$U \mapsto \underrightarrow{\lim}_{V} V$,
where $V$ runs through all quasi-compact open subsets of $U$,
and the inductive limit is taken in the category
of coherent spaces, \textit{not} in the category of 
topological spaces. 
\end{enumerate}
\end{Def}
We are finally in the stage of defining
$\scr{A}$-schemes.
\begin{Def}
\begin{enumerate}
\item An $\scr{A}$-scheme is a triple $(X,\scr{O}_{X},\beta_{X})$
where $X$ is a coherent space,
$\scr{O}_{X}$ is a ring valued sheaf on $X$, and
$\beta_{X}:\alpha_{1}\scr{O}_{X} \to \tau_{X}$
is a morphism of $\cat{IIRng}$-valued sheaves on $X$
(which we refer to as the `support morphism'.
Here, $\alpha_{1}\scr{O}_{X}$ is the sheafification
of $U \mapsto \alpha_{1}\scr{O}_{X}(U)$),
satisfying the following property:
for any two open subsets $U \supset V$ of $X$,
the \textit{restriction maps reflect localizations}, i.e.
the map $\scr{O}_{X}(U) \to \scr{O}_{X}(V)$
factors through $\scr{O}_{X}(U)_{Z}$,
where $Z=U \setminus V$ is a closed subset of $U$
and $\scr{O}_{X}(U)_{Z}$ is the localization
of $\scr{O}_{X}(U)$ along
\[
\{ a \in \scr{O}_{X}(U) \mid \beta_{X}\alpha_{2}(a) \geq Z\}.
\]
\item Let $X=(|X|,\scr{O}_{X},\beta_{X})$
and $Y=(|Y|,\scr{O}_{Y},\beta_{X})$ be two
$\scr{A}$-schemes. A morphism $f:X \to Y$
of $\scr{A}$-schemes is a pair $f=(|f|,f^{\#})$,
where $|f|:|X| \to |Y|$ is a quasi-compact morphism
between underlying spaces,
and $f^{\#}:\scr{O}_{Y} \to f_{*}\scr{O}_{X}$
is a morphism of ring valued sheaves on $Y$ which
makes the following diagram commutative:
\[
\xymatrix{
\scr{O}_{Y} \ar[r]^{f^{\#}} \ar[d]_{\beta_{Y}} &
|f|_{*}\scr{O}_{X} \ar[d]^{\beta_{X}} \\
\tau_{Y} \ar[r]_{|f|^{-1}} & |f|_{*}\tau_{X}
}
\]
\item The spectrum functor 
$\Spec^{\scr{A}}:\cat{Rng} \to \cat{$\scr{A}$-Sch}^{\op}$
from the category of rings to the opposite category of $\scr{A}$-schemes,
is defined as follows:
for a ring $R$, the underlying space is defined by
$X=\Spec R$. The structure sheaf $\scr{O}_{X}$
is the sheafification of $U \mapsto R_{Z}$,
where $Z=X \setminus U$ is the complement closed
subset of $X$, and $R_{Z}$ is the localization along
\[
\{ a \in R \mid (a) \geq Z \}.
\]
The support morphism
$\beta_{X}:\alpha_{1}\scr{O}_{X} \to \tau_{X}$
is the canonical isomorphism. Hence
we set $\Spec^{\scr{A}}R=(X,\scr{O}_{X},\beta_{X})$.

For a homomorphism $f:A \to B$,
we have a morphism $\Spec^{\scr{A}} B \to \Spec^{\scr{A}} A$,
as is well known.
\end{enumerate}
\end{Def}
The spectrum functor is the left adjoint of the global
section functor $\Gamma:\cat{$\scr{A}$-Sch}^{\op} \to \cat{Rng}$.
\begin{Rmk}
\begin{enumerate}
\item There are some differences in the notation with
that of \cite{Takagi}:
in the previous paper, the category of II-rings
is denoted by $\cat{PIIRng}$.
Also, the sheaf $\tau_{X}$ is denoted by $\tau^{\prime}_{X}$.
This is because we are comparing them
with those of sober spaces in \cite{Takagi},
and hence had to distinguish the notation.
However, this is not necessary in this paper.
\item In \cite{Takagi},
presheaves on a coherent space $X$
is defined as a functor $C(X)_{\cpt} \to \cat{Set}$.
In this paper, most of the presheaves are
described in a usual way, namely, we
attach algebras to each \textit{open} subsets of $X$.
However, in some of the definitions and arguments,
we describe sheaves as a functor from $C(X)_{\cpt}$
to simplify the argument. These
two ways of descriptions are essentially the same.
\end{enumerate}
\end{Rmk}

The basic ideas of this paper came up
during the stay in Jussieu University.
We would like to express our gratitude to
their hearty supports during the stay,
especially to Professor V. Maillot
and Professor G. Freixas.
Also, the author owes a lot to
Professor A. Moriwaki 
and colleagues in the HAG seminar, who gave
precious chances of discussions.

\textsc{S. Takagi: Department of Mathematics, Faculty of Science,
Kyoto University, Kyoto, 606-8502, Japan}

\textit{E-mail address}: \texttt{takagi@math.kyoto-u.ac.jp}
\end{document}